\documentclass[10pt,leqno]{article}

\usepackage{amsmath,amssymb,amsthm,mathrsfs,dsfont}

\usepackage[margin=3cm]{geometry} %宽版时用，配合双倍行距

\usepackage{titlesec,hyperref}

\usepackage{color}
\usepackage{fancyhdr}
\titleformat{\subsection}{\it}{\thesubsection.\enspace}{1pt}{}

\newtheorem{theo}{Theorem}[section]
\newtheorem{lemm}[theo]{Lemma}
\newtheorem{defi}[theo]{Definition}
\newtheorem{coro}[theo]{Corollary}
\newtheorem{prop}[theo]{Proposition}
\newtheorem{rema}[theo]{Remark}
\numberwithin{equation}{section}

\allowdisplaybreaks %允许公式分页显示

\newcommand\lm{{\lesssim}}
\begin{document}
\title{Global existence and optimal decay rate of weak solutions to some inviscid Oldroyd-B models
\hspace{-4mm}
}

\author{ Wenjie $\mbox{Deng}^1$ \footnote{E-mail: detective2028@qq.com},\quad
Zhaonan $\mbox{Luo}^2$ \footnote{E-mail: luozhn@fudan.edu.cn}\quad and\quad
Zhaoyang $\mbox{Yin}^{1,3}$\footnote{E-mail: mcsyzy@mail.sysu.edu.cn}\\
$^1\mbox{Department}$ of Mathematics,
Sun Yat-sen University, Guangzhou 510275, China\\
$^2\mbox{School}$ of Mathematical Sciences,
Fudan University, Shanghai 200433, China\\
$^3\mbox{Faculty}$ of Information Technology,\\
Macau University of Science and Technology, Macau, China}
\date{}
\maketitle
\hrule

\begin{abstract}
This paper is devoted to global existence and optimal decay rate of weak solutions to some inviscid Oldroyd-B models with center diffusion. By virtue of the properties of Calderon-Zygmund operator and the Littlewood-Paley decomposition theory, we firstly prove that the 2-D co-rotation inviscid Oldroyd-B model admits global weak solutions with some large data under different integrability conditions. Furthermore, we prove the energy conservation of weak
solutions for the co-rotation case. These obtained results generalize and cover the classical results
for the Euler equation. Moreover, we establish global weak solutions with small data for the 2-D noncorotation inviscid Oldroyd-B model without damping. Finally, we prove optimal decay rate of global weak solutions for the noncorotation case by the improved Fourier splitting method.  \\
%Our strategy relies on the results established in our previous work \cite{Zhang-Yin}.
\vspace*{5pt}

\vspace*{5pt}
\end{abstract}

\vspace*{10pt}

%\phantomsection
%\addcontentsline{toc}{section}{\contentsname}
%添加目录到书签
\tableofcontents
	\section{Introduction}
	In this paper, we firstly introduce the general Oldroyd-B models  \cite{2015Elgindi}:
	\begin{align}\label{eq0}
		\left\{\begin{array}{l}
			\partial_tu + u\cdot\nabla u+\nabla {\rm P} = {\rm div}~\tau+\nu\Delta u,~~~~{\rm div}~u=0,\\[1ex]
			\partial_t\tau + u\cdot\nabla\tau+a\tau+Q(\nabla u,\tau)=\alpha D(u)+\mu\Delta\tau,\\[1ex]
			u|_{t=0}=u_0,~~\tau|_{t=0}=\tau_0. \\[1ex]
		\end{array}\right.
	\end{align}
	In \eqref{eq0}, $u(t,x)$ represents the velocity of the polymeric liquid, $\tau(t,x)$ is the symmetric tensor of constrains and $\rm P$ denotes the pressure. The parameters $a$, $\mu$ and $\nu$ are nonnegative and $\alpha>0$.
	Furthermore, $Q$
	is the following bilinear form
	$$Q(\nabla u, \tau)=\tau \Omega(u)-\Omega(u)\tau+b(D(u)\tau+\tau D(u)),~~~~~~~b\in[-1, 1],$$
	with vorticity tensor
	$$
	\Omega(u)=\frac {\nabla u-(\nabla u)^T} {2},
	$$
	and deformation tensor
	$$
	D(u)=\frac {\nabla u+(\nabla u)^T} {2}.
	$$
	For more explanations of the general Oldroyd-B models \eqref{eq0}, one can refer to \cite{1958Non}.
	
	Taking $b=1$ and $\alpha=2$, then the general Oldroyd-B models \eqref{eq0} can be derived from the following micro-macro models \cite{Doi1988,Diff2DF-P}  with the classical Hookean potential $\mathcal{U}=\frac 1 2 |q|^2$, the drag term $\sigma(u)=\nabla u$ and conservation law $\int_{\mathbb{R}^d}\psi dq = 1$:
	\begin{align}\label{eq1}
		\left\{
		\begin{array}{ll}
			\partial_tu+u\cdot\nabla u+\nabla P ={\rm div}~\tau+\nu\Delta u,~~~~{\rm div}~u=0,\\[1ex]
			\psi_t+u\cdot\nabla\psi={\rm div}_{q}[- \sigma(u)\cdot{q}\psi+\frac a 2\nabla_{q}\psi+\frac a 2\nabla_{q}\mathcal{U}\psi]+\mu\Delta\psi, \\[1ex]
			\tau_{ij}=\int_{\mathbb{R}^{d}}(q_{i}\nabla_{q_j}\mathcal{U})\psi dq-Id, \\[1ex]
			u|_{t=0}=u_0,~~\psi|_{t=0}=\psi_0. \\[1ex]
		\end{array}
		\right.
	\end{align}
	In \eqref{eq1}, the polymer particles are described by the distribution function $\psi(t,x,q)$. Here the polymer elongation $q$ satisfies $q\in\mathbb{R}^d$, which means that the extensibility of the polymers is infinite and $x\in\mathbb{R}^d$. $\tau$ is an extra-stress tensor which generated by the polymer particles effect. In general, $\sigma(u)=\nabla u$, while $\sigma(u)=\Omega$ for the co-rotation case.
		
	The micro-macro models are of great interest in many branches of physics, chemistry, and biology, which describe the system coupling fluids and polymers. A polymer in the models is idealized as an "elastic dumbbell" consisting of two "beads" joined by a spring. At the level of liquid, the system couples the Navier-Stokes equation or the Euler equation for the fluid velocity with a Fokker-Planck equation describing the evolution of the polymer density. (For more details, one can refer to $\cite{Masmoudi2013}$).
	
	The co-rotation inviscid Oldroyd-B model \cite{DLY1} in the following can be derived from the micro-macro models \eqref{eq1} with $\nu=0$, the drag term $\sigma(u)=\Omega$, the Hookean potential $\mathcal{U}=\frac 1 2 |q|^2$ and conservation law $\int_{\mathbb{R}^d}\psi dq = 1$:
	\begin{align}\label{eq2}
		\left\{\begin{array}{l}
			\partial_tu + u\cdot\nabla u+\nabla P={\rm div}~\tau,~~~~{\rm div}~u=0,\\[1ex]
			\partial_t\tau + u\cdot\nabla\tau+a\tau+Q(\Omega,\tau)=\mu\Delta\tau,\\[1ex]
			u|_{t=0}=u_0,~~\tau|_{t=0}=\tau_0,
		\end{array}\right.
	\end{align}
    where $Q$ is the following bilinear form $Q(\Omega, \tau)=\tau \Omega(u)-\Omega(u)\tau$.	

	Let's recall the Euler equation \cite{2002Majda} in the following:
	\begin{align}\label{eq3}
		\left\{\begin{array}{l}
			\partial_t u + u\cdot\nabla u + \nabla {\rm P} = 0,\\[1ex]
			{\rm div}~u=0, ~~~~~ u|_{t=0}=u_0.
		\end{array}\right.
	\end{align}
	
	Notice that the equations \eqref{eq2} reduces to the well-known Euler equation \eqref{eq3} by taking $\tau=0$. However, taking $\tau=0$ in \eqref{eq0}, then we have $D(u)=0$, which implies $u=0$ in Sobolev spaces. The observation reveals the essential difference between \eqref{eq0} and \eqref{eq2}.	
	\subsection{The general Oldroyd-B models.}
	We first recall some mathematic results for the classical Oldroyd-B model $({\rm with}~\nu>0,~\mu=0~)$. In \cite{Guillope1990}, C. Guillop\'e,  and J. C. Saut first showed that the Oldroyd-B model admits a unique global strong solution in Sobolev spaces. Global weak solutions of the Oldroyd-B model was proved by P. L. Lions and N. Masmoudi \cite{Lions-Masmoudi} for the case $\alpha=0$. Notice that the problem for the case $\alpha\neq0$ is still open, see \cite{2011Global,Masmoudi2013}. Later on, J. Y. Chemin and N. Masmoudi \cite{Chemin2001}
	proved the existence and uniqueness of strong solutions in homogenous Besov spaces with critical index of regularity. Optimal decay rates for solutions to the 3-D Oldroyd-B model were obtained by M. Hieber, H. Wen and R. Zi \cite{2019OldroydB}. An approach based on the deformation tensor can be found in \cite{Li2,Li1,Lei2008,2010On1,2010On}.
	
	Some mathematic results for the inviscid Oldroyd-B models with center diffusion $({\rm for}~\nu=0,~\mu>0~)$ are given as follows. T. M. Elgindi and F. Rousset \cite{2015Elgindi} first proved regularity for the 2-D Oldroyd-B models \eqref{eq0} with $a>0$. Later on, T. M. Elgindi and J. Liu \cite{2015Elgindi1} obtained the global strong solutions of the 3-D case under the assumption that initial data is sufficiently small. Recently, W. Deng, Z. Luo and Z. Yin \cite{DLY1} obtained the global solutions to \eqref{eq2} in co-rotation case and proved the $H^1$ decay rate for the global solutions constructed by T. M. Elgindi and F. Rousset. For the case $a=0$, P. Constantin, J. Wu, J. Zhao and Y. Zhu \cite{P.Constantin} established the global well-posedness of the inviscid Oldroyd-B models for fractional dissipation with small data. In \cite{Wu}, J. Wu and J. Zhao investigated the global well-posedness in Besov spaces for fractional dissipation with small data. Optimal time decay rate in $H^1$ framework of global strong solutions to the inviscid Oldroyd-B models for fractional dissipation was given by \cite{Wu1}.
	However, they can't deal with critical case for \eqref{eq0} with $d=2$ and integer dissipation.
	\subsection{The Hooke dumbbell models.}
	Let $\nu,\mu>0$. The construction of global weak solutions for micro-macro systems was considered in \cite{Hookeweak1,Hookeweak5,Hookeweak6}. Recently, the so-called moments $(u,M_{a,b})$ for the diffusive 2-D models considered in \cite{Diff2DF-P} are strong solutions with macroscopic variables $(t,x)$ while $\psi$ is nonnegative measures on $\mathbb{R}^2_q\times\mathbb{R}^2_x$ merely.
	
	Let $\mu=0$. The local existence of micro-macro systems were proved by many researchers in different settings, see \cite{2004Well,Renardy1989An}. F. H. Lin, C. Liu and P. Zhang \cite{Zhang2007} studied the incompressible micro-macro polymeric system and proved global existence near equilibrium for some assumptions on the potential $\mathcal{U}$. The global regularity for the 2-D co-rotation Hooke dumbbell model was proved by N. Masmoudi, P. Zhang, and Z. Zhang \cite{HookeGlobal}. Notice that the problem for the noncorotation case $\sigma(u)=\nabla u$ is still open, see \cite{Masmoudi2013}. The long time behavior for the 3-D micro-macro polymeric system was considered by L. He and P. Zhang \cite{He2009}.
	\subsection{Global weak solutions of the Euler equation.}
	We firstly recall a definition and the main result of weak solutions for the Euler equation \eqref{eq3}.
	\begin{defi}
		A velocity field $v(x,t)$ with initial data $v_0$ is a weak solution of the Euler equation in primitive-variable
		form if the following conditions hold:\\
		(1) $v\in L^1([0,T]\times B_R)$ with $B_R=\{x\in\mathbb{R}^2,|x|\leq R\},~~~\forall~R\in(0,\infty)$,\\
		(2) $v\otimes v= (v_i v_j)\in L^1([0,T]\times B_R)$,\\
		(3) ${\rm div}~v =0$ in the sense of distributions, i.e.,
		\begin{align*}
			\int_{\mathbb{R}^2}\nabla \psi \cdot v =0,~~~\forall ~\psi \in C(0,T;C^1_0(\mathbb{R}^2)),
		\end{align*}
		(4) for any $\phi=(\phi_1,\phi_2)\in C(0,T;C^1_0(\mathbb{R}^2))$ with ${\rm div}~\phi =0$,
		\begin{align*}
			\int_{\mathbb{R}^2}\phi(x,T)\cdot v(x,T)dx - 	\int_{\mathbb{R}^2}\phi(x,0)\cdot v_0(x)dx=\int_0^T\int_{\mathbb{R}^2} (\phi_t\cdot v + \nabla\phi:(v\otimes v))dxdt.
		\end{align*}
	\end{defi}
	\begin{theo}\label{Euler}\cite{2002Majda}
		Let $p\in(1,\infty]$. Assume that $u_0$ is a divergence-free field and vorticity $\Omega_0={\rm curl}~u_0\in L^1\cap L^p$, then \eqref{eq3} admits a global weak solution $u\in L^{\infty}_T(L_{loc}^2\cap\dot{W}^{1,p})$.
	\end{theo}
	\subsection{Main results.}
	Global existence and long time behaviour of weak solutions for polymeric models were concerned by N. Masmoudi \cite{Masmoudi2013,2016Equations}. To our best knowledge, global existence and large time behaviour of weak solutions for the 2-D inviscid Oldroyd-B models have not been studied yet. In this paper, we firstly study global weak solutions of \eqref{eq2} under different integrability conditions. By virtue of energy estimates and the improved Fourier spiltting method, we then derive global existence and optimal decay rate of all order spatial derivatives for weak solutions to \eqref{eq0} with $\nu=a=0$.
	
	Firstly, we give a definition of weak solutions to the 2-D co-rotation inviscid Oldroyd-B model \eqref{eq2}.
	\begin{defi}\label{defi}
		Suppose $\psi$ and $\phi\in\mathscr{D}([0,T)\times\mathbb{R}^2)$ with ${\rm div}~\psi=0$, then we say $(u,\tau)$ is a global weak solution for \eqref{eq2} if the following conditions hold: \\
		(1) For any $T>0$, $u\in C_T(L^2)\cap L^{\infty}_T(\dot{W}^{1,p})$ with $p\in(1,\infty)$ and $\tau\in C_T(L^2_{w})\cap L^{\infty}_T(L^2)\cap L^2_T(\dot{H}^{1})$.\\
		(2) For any $\psi$ and $\phi$, there holds
		\begin{align*}
			&\int_{0}^{t}\int_{\mathbb{R}^2} u\psi_t + (u\otimes u):\nabla\psi - \tau:\nabla\psi dxdt = -\int_{\mathbb{R}^2} u_0\psi(0,x) dx,\\ \notag
			&\int_{0}^{t}\int_{\mathbb{R}^2} \tau\phi_t + (u\otimes \tau):\nabla\phi - a\tau\phi - Q(\Omega,\tau)\phi - \mu\nabla\tau:\nabla\phi dxdt = -\int_{\mathbb{R}^2} \tau_0\phi(0,x) dx.
		\end{align*}
	\end{defi}
	
	Our main results can be stated as follows.
	\begin{theo}\label{theo}
		Let $d=2$ and $a=\mu=1$. Assume a divergence-free field $u_0\in H^1$ and a symmetric matric $\tau_0\in H^1\cap L^\infty$, then  \eqref{eq2} admits a global weak solution $(u,\tau)\in L^{\infty}(0,\infty;H^{1})$ satisfying
		\begin{align*}
			\|(u,\tau)\|_{H^1} \leq \|(u_0,\tau_0)\|^2_{H^1} e^{C(1+\|\tau_0\|^2_{L^{\infty}}+\|\tau_0\|^2_{L^2})}.
		\end{align*}
	\end{theo}
	\begin{rema}
		If we add the smallness of $\|\nabla u_0\|_{L^2}$ and $\|\tau_0\|_{H^1}$, then the boundness condition $\tau_0\in L^\infty$ in Theorem \ref{theo} can be removed, see \cite{DLY1}.
	\end{rema}
	\begin{theo}\label{theo1}
		Let $d=2$ and $a=\mu=1$. \\
		(1). Let $p\in(1,2)$. Assume that a divergence-free field $u_0\in L^2 \cap \dot{W}^{1,p}$ and a symmetric matrix $\tau_0\in L^p\cap L^2$. There exists some sufficiently small positive constant $c$, which is not dependent on the initial data, such that if
		\begin{align}\label{con1}
		\|\tau_0\|_{L^2}\leq \frac {c} {1+\|\tau_0\|_{L^p}+\|\nabla u_0\|_{L^p}},
		\end{align}
		then \eqref{eq2} admits a global weak solution $(u,\tau)$ with $$
		u\in L^{\infty}(0,\infty;L^2\cap\dot{W}^{1,p}),~~~\tau \in L^{\infty}(0,\infty;L^p\cap L^2)\cap (L^1\cap L^2)(0,\infty; \dot{H}^1)\cap L^1(0,\infty; B^0_{\infty,1}).
		$$
		Moreover, if additionally $(u_0,\tau_0)\in \dot{W}^{1,r}\times L^{r},$
		with $r\in[2,\infty)$, one can arrive at
		$$
		u \in L^{\infty}(0,\infty;\dot{W}^{1,r}),~~~\tau \in L^{\infty}(0,\infty;L^{r}).
		$$
		(2). Let $p\in[2,\infty)$. Assume that a divergence-free field $u_0\in L^2 \cap \dot{W}^{1,p}$ and a symmetric matrix $\tau_0\in L^{2}\cap L^p$. There exists some sufficiently small positive constant $c$, which is not dependent on the initial data, such that if
		\begin{align}\label{con2}
		\|\tau_0\|_{L^2}\leq \frac {c} {1+\|\tau_0\|_{L^p}+\|u_0\|_{L^2\cap\dot{W}^{1,p}}},
		\end{align}
		then \eqref{eq2} admits a global weak solution $(u,\tau)$ with
		$$
		u \in L^{\infty}(0,\infty;L^2\cap\dot{W}^{1,p}),~~~\tau \in L^{\infty}(0,\infty;L^{2} \cap L^p)\cap (L^{1}\cap L^{2})(0,\infty;\dot{H}^1).
		$$
		Moreover, if $p\in(2,\infty)$, one can arrive at
		\begin{align*}
		\|\tau\|_{\tilde{L}^1(0,\infty;B^{2-\frac{2}{p}}_{\infty,\infty})} \leq C\left(\|u_0\|_{L^2\cap \dot{W}^{1,p}} + \|\tau_0\|_{L^2\cap L^p}\right).
		\end{align*}
	\end{theo}
	\begin{rema}
		Our results on global weak solutions cover the cases for the incompressible
		Euler equation by taking $\tau = 0$ in Theorem \ref{theo} and Theorem \ref{theo1}. Moreover, we prove the energy conservation for weak
		solutions of \eqref{eq2} (see Proposition \ref{onsager} and Remark \ref{onsager1}), which cover the Onsager's conjecture on the energy conservation \cite{On1} for the Euler equation.
	\end{rema}
	\begin{rema}
		Taking $p=2$ in Theorem \ref{theo1}, we obtain a global weak solution of \eqref{eq2} without $\tau_0\in L^{\infty}$ and $\nabla\tau_0\in L^2$. Compared with Theorem \ref{theo}, this result reduces the requirements for regularity and integrability of $\tau$.
	\end{rema}
	\begin{rema}
		Let $p\in(1,2]$ and $u_0 =A(x_2e^{-|x|},-x_1e^{-|x|})$.
		We can verify that ${\rm div}~u_0 = 0$ and
		$$
		\|u_0\|_{L^2\cap\dot{W}^{1,p}} \leq C.
		$$
		Suppose that $h(x) =Ae^{-|x|^2}\frac {1}{|x|\ln(e+|x|^{-1})}$, then we have $\|h\|_{L^{\infty}} =\infty$ and $\|\nabla h\|_{L^{2}}=\infty$.
		
		(1) If $p\in(1,2)$, we take a symmetric matrix $\tau_0 =\varepsilon^{\frac 2 p} h(\varepsilon x)\cdot{\rm Id}.$
		One can see that
		$$
		\|\tau_0\|_{L^p} \approx \|h\|_{L^p},~~~~\|\tau_0\|_{L^2} \leq C \varepsilon^{\frac 2 p-1}.
		$$
		
		(2) If $p=2$, we take $\tau_0 =\varepsilon^{2} h(\varepsilon x)\cdot{\rm Id}.$
		One can see that
		$$
		\|\tau_0\|_{L^2} \leq C \varepsilon.
		$$
		Notice that $\|\tau_0\|_{L^{\infty}} =\infty$ and $\|\nabla\tau_0\|_{L^{2}}=\infty$. We can construct large initial data
		$$
		(u_0,\tau_0)\in(L^2\cap \dot{W}^{1,p})\times(L^p\cap L^2),
		$$ by taking $A$ sufficiently
		large. Suppose that $\varepsilon$ is very small relative to $A$, then we obtain some initial data that satisfies the conditions in Theorem \ref{theo1} with $p\in(1,2]$.
	\end{rema}
	\begin{rema}
		Let $p\in(2,\infty)$ and $u_0 =A(x_2e^{-|x|},-x_1e^{-|x|})$.
		Then, we have ${\rm div}~u_0 = 0$ and
		$$
		\|u_0\|_{L^2\cap\dot{W}^{1,p}} \leq C.
		$$
		Define $h(x) =\frac {A}{(1+|x|)\ln(e+|x|)}$ and $\tau_0 =\varepsilon^{-\frac 2 p} h(\varepsilon^{-1} x)\cdot{\rm Id}.$
		One can see that
		$$
		\|\tau_0\|_{L^p} \approx \|h\|_{L^p},~~~~\|\tau_0\|_{L^2} \leq C \varepsilon^{1-\frac 2 p}.
		$$
		We can construct large initial data
		$$
		(u_0,\tau_0)\in(L^2\cap \dot{W}^{1,p})\times(L^p\cap L^2),
		$$ by taking $A$ sufficiently
		large. Suppose that $\varepsilon$ is very small relative to $A$, then we obtain some initial data that satisfies the conditions in Theorem \ref{theo1} with $p\in(2,\infty)$.
	\end{rema}
	Then, we give a definition of weak solutions to the noncorotation inviscid Oldroyd-B model \eqref{eq0} with $\nu=0$.
	\begin{defi}
		Suppose $\psi$ and $\phi\in\mathscr{D}([0,T)\times\mathbb{R}^2)$ with ${\rm div}~\psi=0$. then we say $(u,\tau)$ is a global weak solution for \eqref{eq0} with $\nu=0$ if the following conditions hold: \\
		(1) $\forall~T>0$, $u\in L^{\infty}_T(H^1)$ and $\tau\in L^{\infty}_T(H^1)$.\\
		(2) For any $\psi$ and $\phi$, there holds
		\begin{align*}
		&\int_{0}^{t}\int_{\mathbb{R}^2} u\psi_t + (u\otimes u):\nabla\psi - \tau:\nabla\psi dxdt = -\int_{\mathbb{R}^2} u_0\psi(0,x) dx,\\ \notag
		&\int_{0}^{t}\int_{\mathbb{R}^2} \tau\phi_t + (u\otimes \tau):\nabla\phi - a\tau\phi + \alpha D(u)\phi- Q(\nabla u,\tau)\phi - \mu\nabla\tau:\nabla\phi dxdt = -\int_{\mathbb{R}^2} \tau_0\phi(0,x) dx.
		\end{align*}
	\end{defi}	
	\begin{theo}\label{theo2}
		Let $d=2$, $b=\alpha=\mu=1$ and $a=\nu=0$. Assume a divergence-free field $u_0\in H^1$ and a symmetric matric $\tau_0\in H^1$. There exists some positive constant $c$ small enough such that if
	    \begin{align}
		\|(u_0,\tau_0)\|_{H^1} \leq c,
		\end{align}
		then \eqref{eq0} admits a global weak solution $(u,\tau)$ with
		$$
		(u,\tau) \in L^{\infty}(0,\infty;H^1).
		$$
		Moreover, if additionally $(u_0,\tau_0) \in \dot{B}^{-1}_{2,\infty},$
		then there exists a positive constant $C$ such that
		$$
		\|(u,\tau)\|_{L^2} +(1+t)^{\frac{1}{2}}\|\nabla(u,\tau)\|_{L^2}\leq C(1+t)^{-\frac{1}{2}}.
		$$
	\end{theo}
	\begin{rema}
		Combining with the result of lower bound estimate in \cite{LLY}, one can see that the $H^1$ decay rate for $(u,\tau)$ obtained in Theorem \ref{theo2} is optimal.
	\end{rema}
	\begin{rema}
		The classical result about large time behaviour often supposed that the initial data belongs to $L^1$, see \cite{Schonbek1985}.
		Since $L^1\hookrightarrow \dot{B}^{-1}_{2,\infty}$, it follows that the above results still hold true when $(u_0,\tau_0)\in L^1$.
	\end{rema}
	\subsection{Motivations and main ideas.}
		As is well known, the vorticity equation derived from the 2-D Euler equation propagates along the velocity field, that is
		\begin{align}
		\frac{\partial}{\partial t}\Omega + u\cdot\nabla\Omega = 0.
		\end{align}
		The transport structure of the vorticity equation ensures that the vorticity satisfies conservation law of $L^p$ norm, where $p\in[1,\infty]$. By virtue of the conservation laws, global existence of weak solutions for the 2-D Euler equation were established under different integrability conditions, see Theorem \ref{Euler}.
		
		The vorticity equation derived from the 2-D inviscid Oldroyd-B model \eqref{eq2} satisfies the following form:
		\begin{align}\label{idea1}
		\frac{\partial}{\partial t}\Omega + u\cdot\nabla\Omega = \nabla\times{\rm div}~\tau,
		\end{align}
		where the external force term $\nabla\times{\rm div}~\tau$ disrupts the conservation laws of the transport structure. Therefore, global existence of weak solutions for the inviscid Oldroyd-B model \eqref{eq2} is a challenging problem.
		
		In special case $p=2$, if $(u_0,\tau_0)\in H^1$ and $\tau_0\in L^\infty$, one can derive a global energy estimate in $L^{\infty}(0,\infty;H^{1})$ by virtue of exponential decay of $\|\tau\|_{L^\infty}$, see \cite{DLY1}. Applying the compactness method, we obtain global weak solutions of \eqref{eq2} without any small condition.
		
		For general $p$, a direct observation of \eqref{idea1} reveals that when $\tau$ is sufficiently small, the behavior of the hyperbolic system \eqref{idea1} should be similar to the behavior of the 2-D Euler equation. Then, we conjecture that global estimate can be obtained when $\tau$ is sufficiently small. However, if we directly derive the estimate of $\|\Omega\|_{L^p}$ from \eqref{idea1}, we need to establish the estimate of
		$$
		\int_0^t\|\nabla\times{\rm div}\tau\|_{L^{p}}ds.
		$$
		Notice that $B^{2}_{p,1}\hookrightarrow W^{2,p}$. Applying the Littlewood-Paley decomposition theory of the transport-diffusion equation to $(\ref{eq2})_2$, we have to control $\int_0^t\| Q(\Omega,\tau)\|_{B^0_{p,1}}ds$. Unfortunately, by Bony's decomposition, we deduce that
		$$
		\int_0^t\|T_{\tau}\Omega\|_{B^0_{p,1}}ds \leq C\int_0^t\|\tau\|_{L^{\infty}}\|\Omega\|_{B^{0}_{p,1}}ds\leq C\|\Omega\|_{L_t^{\infty}(B^{0}_{p,1})}\int_0^t\|\tau\|_{L^{\infty}}ds.
		$$
		The unclosed estimate forces us to establish the estimate of $\|\Omega\|_{B^{0}_{p,1}}$, instead of  $\|\Omega\|_{L^p}$ at the very beginning. Applying the refined estimate in Besov spaces with index $0$, we infer \eqref{idea1} from that
		\begin{align*}
			\|\Omega\|_{L^{\infty}_t (B^0_{p,1})} \leq C	(\|\Omega_0\|_{B^0_{
					p,1}} + \|\nabla\times{\rm div}~\tau\|_{L^1_t(B^0_{p,1})})(1 + \int_0^t \|\nabla v\|_{L^{\infty}}ds).
		\end{align*}
	    The commutator estimate would lead to the appearance of $\|\nabla u\|_{L^{\infty}}$, which requires high regularity and integrability of initial data and is not within the framework of weak solutions.

		In order to obtain global existence of weak solutions, we need to utilize the special structure and intrinsic properties of \eqref{eq2}. To cancel $\nabla \times{\rm div}~\tau$ and $\Delta\tau$, we introduce the structural trick as follows:
		\begin{align}\label{idea2}
		\Gamma = \mu\Omega(u)-\mathscr{R}\tau,
	    \end{align}
		where $\mathscr{R}=-(\Delta)^{-1}{\rm curl}~{\rm div}$. From \eqref{eq2}, we derive that the structural trick $\Gamma$ satisfies the following hyperbolic system:
		\begin{align}\label{idea3}
			\partial_t\Gamma + u\cdot\nabla\Gamma = a\mathscr{R}\tau + \mathscr{R}Q(\Omega,\tau) + [\mathscr{R},u\cdot\nabla]\tau.
		\end{align}
		Different from \cite{2015Elgindi}, there is no damping phenomenon for $\Gamma$ or $\Omega$ for lack of $D(u)$. We point that the cancellation between ${\rm div}~\tau$ and $\Delta\tau$ by introducing $\Gamma$ is crucial to establish global estimates of weak solutions for \eqref{eq2}. Taking $L^p$ norm to \eqref{idea3}, we derive the estimate of $\|\Gamma\|_{L^p}$, which is equivalent to $\|\nabla u\|_{L^p}$ with $1<p<\infty$ and the conservation laws of $\tau$. By virtue of the properties of Calderon-Zygmund operator for $1<p<\infty$, we can deal with the first two terms on the right side of the equation \eqref{idea3}. If $p\in[2,\infty)$, we can deal with the last term on the right side of the equation \eqref{idea3} by the commutator lemma, see \cite{2011Hmidi}. For the case $p\in(1,2)$, we prove a new commutator estimate for $\|[\mathscr{R},u\cdot\nabla]\tau\|_{L^p}$ by the Littlewood-Paley decomposition theory. The commutator estimates enable us to derive a global estimate for $\|\Gamma\|_{L^p}$ by virtue of the boundness for $$\int_0^t\|\tau\|_{B^0_{\infty,1}}ds.$$
		However, $\int_0^t\|\tau\|_{B^0_{\infty,1}}ds$ can not be controlled by basic energy estimates for $\tau$. To solve this difficulty, we need to fully utilize the parabolic effect of $(\ref{eq2})_2$. According to Bony's decomposition and a special energy estimate $\tau \in L^{1}(0,\infty;H^1)$, we derive the estimate for $\int_0^t\|\tau\|_{B^0_{\infty,1}}ds$, which is controlled by $\|\Gamma\|_{L^p}$ and initial data $\|\tau_0\|_{L^2}$.	
		Suppose that the low frequency condition $\|\tau_0\|_{L^2}$ small enough, then we obtain a closed global estimate of $\|\Gamma\|_{L^p}$ for $1<p<\infty$. Applying the compactness method, we prove global existence of weak solutions for \eqref{eq2} with some large data. Our results on global weak solutions cover the cases for the well-known Euler equation by taking $\tau = 0$.
		Moreover, for $p>\frac 3 2$, we prove the energy conservation for weak solutions of \eqref{eq2} by virtue of the polishing approximation method and a characterization of Besov spaces $\dot{B}^{\alpha}_{3,\infty}$ with $\alpha\in(0,1)$.
		Considering the inviscid MHD equation \cite{LZ1,Wei2020}, we fail to establish global estimates of weak solutions with $\Omega\in L^p$ for lack of the structural trick $\Gamma$ to cancel the external force term of the vorticity equation derived from the inviscid MHD equation. This is an interesting and challenging problem.

		For the 2-D noncorotation inviscid Oldroyd-B model \eqref{eq0}  $({\rm with}~\nu=0)$, T. M. Elgindi and F. Rousset \cite{2015Elgindi} first proved global existence for weak solutions the with small data in $H^1$ by virtue of the structural trick $\Gamma$ and the damping  $({\rm i.e.}~ a>0)$. Large time behavior of global weak solutions constructed by T. M. Elgindi and F. Rousset was studied by damping effect, see \cite{DLY1}. The authors failed to obtain optimal decay rate of the $\dot{H}^{1}$ norm for lack of external higher order regularity to the solutions.
		
		Considering a more general situation for \eqref{eq0} without the viscosity and the damping $({\rm i.e.}~ \nu=a=0)$, we firstly establish global energy estimates for the solutions with small data in $H^1$ and $d=2$.
		Observing that there exists parabolic effect of the stress tensor $\tau$, we obtain
		\begin{align}\label{E1}
		\frac{d}{dt}\|(u,\tau)\|^2_{L^2} + \|\nabla\tau\|^2_{L^2}
		\leq C\|\tau\|^2_{L^2}\|\nabla u\|^2_{L^2},
		\end{align}
		Considering velocity $u$, we fail to derive dissipative effect from the structural trick $\Gamma$ for lack of the damping. With the small assumption, we introduce inner product estimate instead of the estimate for the structural trick $\Gamma$ as follows
         \begin{align}\label{E2}
         -\frac{d}{dt}\langle\tau,\nabla u\rangle + \frac 1 4\|\nabla u\|^2_{L^2}
         \leq  C\|\nabla\tau\|^2_{H^1}.
         \end{align}
         In addition, the fact $ \langle  u\cdot\nabla u,\Delta u\rangle=0$ for $d=2$ and ${\rm div}~u=0$ is crucial to obtain the closed energy estimate in $H^1$ framework (see Section 3):
        \begin{align}\label{E3}
        \frac{d}{dt}\|\nabla(u,\tau)\|^2_{L^2} + \|\nabla^2\tau\|^2_{L^2} \leq C\|\nabla u\|^2_{L^2}\|\tau\|^2_{H^1}.
        \end{align}
        Applying the compactness method, we prove global existence of weak solutions for \eqref{eq0} with $\nu=0$.

        Finally, we prove optimal decay rate of global weak solutions for the noncorotation case by the improved Fourier splitting method.
		The corresponding linearized system \cite{P.Constantin} of the noncorotation inviscid Oldroyd-B model is given by
		\begin{align}
			\left\{\begin{array}{l}
				\partial_t u = \mathbb{P}{\rm div}~\tau,~~~{\rm div}~u=0,\\[1ex]
				\partial_t \mathbb{P}{\rm div}~\tau - \Delta\mathbb{P}{\rm div}~\tau = \Delta u,
			\end{array}\right.
		\end{align}
		where $\mathbb{P}$ denotes the Leray projection over divergence-free vector fields. One can see that $(u,\mathbb{P}{\rm div}~\tau)$ satisfies the same system of the following wave-type equation:
		\begin{align}
			\partial_t W - \Delta\partial_t W - \Delta W =0,
		\end{align}
		which reveals that there are both dissipative and dispersive effects on $(u,\mathbb{P}{\rm div}~\tau)$. However, taking into account the nonlinear term as well, optimal decay rate under the framework of weak solutions is more difficult than that under the framework of strong solutions. One can see that damping effect for high frequency of velocity $u$ merely appears in \eqref{E2}, while parabolic effect of stress tensor $\tau$ presents in \eqref{E1} and \eqref{E3}. According to \eqref{E1}-\eqref{E3}, there exists strong coupling effect between different frequencies of the weak solutions.

		The main difficulties in proving the optimal decay of weak solutions are as follows.\\
		{\rm\textbf{1.}} Since $d=2$ is a critical case, it was not possible to obtain any algebraic decay rate at the beginning by virtue of the Fourier splitting method.\\
		{\rm\textbf{2.}} Due to strong coupling effect between different frequencies of the weak solutions, we can not directly obtain any algebraic decay rate of $\dot{H}^1$ norm of velocity $u$ without enough regularity for the solutions.\\
		{\rm\textbf{3.}} We can not improve algebraic decay rate of the solutions for lack of the damping and algebraic decay rate of the $\dot{H}^1$ norm for $(u,\tau)$.
		
		To overcome these difficulties, we introduce three essential laws in the following.\\
		{\rm\textbf{ Law 1 :}} The algebraic decay rate of the solutions can provide the information about the boundness for the solutions in Besov spaces with some negative indicators, which also reflects the low-frequency information of the solutions (see Lemma \ref{5lemma2}).\\
		{\rm\textbf{ Law 2 :}} The boundness of the solutions in Besov spaces with negative indicators can further enhance the decay rate (see Lemma \ref{5lemma3}).\\
		{\rm\textbf{ Law 3 :}} In the process of optimal decay rate of weak solutions, for any $k>0$ and $t>0$, the time weighted integral
		$$\int_0^t (1+t)^{k}\|\nabla(u,\tau)\|^2_{L^2}ds \leq C$$
		has the same effect as the decay property
		$$
		(1+t)^{k+1}\|\nabla(u,\tau)\|^2_{L^2} \leq C.
		$$

		Our strategy is as follows:\\
		{\rm\textbf{ Step 1 :}} By virtue of the improved Fourier splitting method, we deduce the logarithmic decay rates in the following:	
		\begin{align}
			\ln^{l}(e+t)\|(u,\tau)\|^2_{H^1} + \int_0^t \ln^{l}(e+s)\|\nabla (u,\tau)\|^2_{L^2}ds\leq C,~~~\forall~l\in\mathbb{N}.
		\end{align}
		Then, we study how to increase the logarithmic decay rate to the algebraic decay rate under the framework of weak solutions.\\
		{\rm\textbf{ Step 2 :}} According to {\rm\textbf{Law 3}} and time weighted energy estimate, we have the original algebraic decay rate
		\begin{align}
			(1+t)^{\frac{1}{2}}\|(u,\tau)\|^2_{H^1} + (1+t)^{1}\|\nabla (u,\tau)\|^2_{L^2} \leq C.
		\end{align}
		Moreover, we obtain the improved time weighted integrability
		$$
		(1+t)^{-1}\int_0^t(1+s)^{\frac{3}{2}}\|\nabla (u,\tau)\|^2_{L^2}ds \leq C,
		$$
		which is useful for proving the boundness of the solutions in Besov space with negative indicators.\\
		{\rm\textbf{ Step 3 :}} According to {\rm\textbf{Law 1}} and {\rm\textbf{Law 3}}, for some $\sigma\in(0,1]$, we deduce that
		$$
		(u,\tau)\in\dot{B}^{-\sigma}_{2,\infty}(\mathbb{R}^2),
		$$
		which provides the low-frequency information of the solutions.\\
		{\rm\textbf{ Step 4 :}} According to {\rm\textbf{Law 2}} and {\rm\textbf{Law 3}}, we enhance the original decay rate to
		$$
		(1+t)^{\gamma}\|(u,\tau)\|^2_{L^2} + (1+t)^{2\gamma}\|\nabla (u,\tau)\|^2_{L^2} \leq C,
		$$
		where $\frac 1 2<\gamma\leq 1$. Notice that when $\gamma=1$, the above time decay of the solutions in $H^1$ is optimal.
		Moreover, we obtain the improved time weighted integrability
		$$
		(1+t)^{-1}\int_0^t (1+s)^{\gamma+1}\|\nabla (u,\tau)\|^2_{L^2}ds \leq C.
		$$
		{\rm\textbf{Step 5 :}} Iterating on {\rm\textbf{Step 3}} and {\rm\textbf{Step 4}}, we ultimately obtain
		$$
		(u,\tau)\in\dot{B}^{-1}_{2,\infty}(\mathbb{R}^2)
		$$
		and the optimal decay rate
		$$
		(1+t)^{1}\|(u,\tau)\|^2_{L^2} + (1+t)^{2}\|\nabla (u,\tau)\|^2_{L^2} \leq C.
		$$
{\rm\textbf{Structure of the paper:}}~~In Section 2, we give some preliminaries which will be used in the sequel. In Section 3, by virtue of the properties of Calderon-Zygmund operator and the Littlewood-Paley decomposition theory, we establish global estimates with some large data for the 2-D inviscid Oldroyd-B models. In Section 4, we prove that the 2-D inviscid Oldroyd-B models admits a global weak solution under different integrability conditions and study the energy conservation of weak
solutions for the co-rotation case.
In Section 5, we study optimal decay rate of global weak solutions in $H^1$ to the 2-D noncorotation inviscid Oldroyd-B model by virtue of the improved Fourier splitting method.

	\section{Preliminaries}
	In this section, we introduce some notations and useful lemmas which will be used in the sequel.
	
	We agree that $f\lm g$ represents $f\leq Cg$ with a constant $C>0$. The symbol $\widehat{f}=\mathscr{F}(f)$ stands for the Fourier transform of $f$. For $p\in[1,\infty]$, $\|\cdot\|_{L^p}$ denotes the norm in the Lebesgue space $L^p$. In addition, $u\in\dot{W}^{1,p}$ is equivalent to $\nabla u\in L^p$. The symbol $\dot{H}^{1}=\dot{W}^{1,2}$.
	
	The Littlewood-Paley decomposition theory is given as follows.
	\begin{prop}\cite{Bahouri2011}\label{prop0}
		Let $\mathscr{C}$ be the annulus $\{\xi\in\mathbb{R}^d:\frac 3 4\leq|\xi|\leq\frac 8 3\}$. There exist radial functions $\chi$ and $\varphi$, valued in the interval $[0,1]$, belonging respectively to $\mathscr{D}(B(0,\frac 4 3))$ and $\mathscr{D}(\mathscr{C})$, and such that
		$$ \forall\xi\in\mathbb{R}^d,\ \chi(\xi)+\sum_{j\geq 0}\varphi(2^{-j}\xi)=1, $$
		$$ \forall\xi\in\mathbb{R}^d\backslash\{0\},\ \sum_{j\in\mathbb{Z}}\varphi(2^{-j}\xi)=1,~~~ $$
		$$ |j-j'|\geq 2\Rightarrow\mathrm{Supp}\ \varphi(2^{-j}\cdot)\cap \mathrm{Supp}\ \varphi(2^{-j'}\cdot)=\emptyset, $$
		$$ ~~j\geq 1\Rightarrow\mathrm{Supp}\ \chi(\cdot)\cap \mathrm{Supp}\ \varphi(2^{-j}\cdot)=\emptyset. $$
		The set $\widetilde{\mathscr{C}}=B(0,\frac 2 3)+\mathscr{C}$ is an annulus, then
		$$ |j-j'|\geq 5\Rightarrow 2^{j}\mathscr{C}\cap 2^{j'}\widetilde{\mathscr{C}}=\emptyset. $$
		Moreover, we have
		$$ ~~\forall\xi\in\mathbb{R}^d,\ \frac 1 2\leq\chi^2(\xi)+\sum_{j\geq 0}\varphi^2(2^{-j}\xi)\leq 1, $$
		$$ \forall\xi\in\mathbb{R}^d\backslash\{0\},\ \frac 1 2\leq\sum_{j\in\mathbb{Z}}\varphi^2(2^{-j}\xi)\leq 1.~~ $$
	\end{prop}
			
		Let $u$ be a tempered distribution in $\mathscr{S}'(\mathbb{R}^d)$. For all $j\in\mathbb{Z}$, define
		$$
		\Delta_j u=0\,\ \text{if}\,\ j\leq -2,\quad
		\Delta_{-1} u=\mathscr{F}^{-1}(\chi\mathscr{F}u),$$
		$$\Delta_j u=\mathscr{F}^{-1}(\varphi(2^{-j}\cdot)\mathscr{F}u)\,\ \text{if}\,\ j\geq 0,\quad
		S_j u=\sum_{j'<j}\Delta_{j'}u.
		$$
		The homogeneous operators are defined by
		$$\dot{\Delta}_j u=\mathscr{F}^{-1}(\varphi(2^{-j}\cdot)\mathscr{F}u).$$
		
		Let $s\in\mathbb{R}$ and $(p,r)\in[1,\infty]^2$. The nonhomogeneous $\rm Besov$ Space $B^s_{p,r}$ is defined by
		\begin{align*}
			B^s_{p,r} = \{u\in\mathscr{S}':\|u\|_{B^s_{p,r}} = \left\|2^{js}\|\Delta_ju\|_{L^p}\right\|_{l^r({\mathbb{Z}})}<\infty\}.
		\end{align*}
	The homogeneous Besov space $\dot{B}^s_{p,r}$ is given as follows
	$$ \dot{B}^s_{p,r}=\{u\in \mathscr{S}':\|u\|_{\dot{B}^s_{p,r}}=\Big\|(2^{js}\|\dot{\Delta}_j u\|_{L^p})_j \Big\|_{l^r(\mathbb{Z})}<\infty\}.$$
		
		For any positive time $T$, the Time-Space Besov Spaces are defined by
		\begin{align*}
			L^{\rho}_T(B^s_{p,r}) = \{u\in\mathscr{S}':\|u\|_{L^{\rho}_T(B^s_{p,r})} = \left\|\left\|2^{js}\|\Delta_ju\|_{L^p}\right\|_{l^r({\mathbb{Z}})}\right\|_{L^{\rho}_T}<\infty\},
		\end{align*}
		and
		\begin{align*}
			\tilde{L}^{\rho}_T(B^s_{p,r}) = \{u\in\mathscr{S}':\|u\|_{\tilde{L}^{\rho}_T(B^s_{p,r})} = \left\|2^{js}\|\Delta_ju\|_{L^{\rho}_T(L^p)}\right\|_{l^r({\mathbb{Z}})}<\infty\}.
		\end{align*}
		Moreover, the following embedding relationships hold:
		$$
		L^{\rho}_T(B^s_{p,r})\hookrightarrow \tilde{L}^{\rho}_T(B^s_{p,r}) ~~~~\text{if}~~~~r\geq\rho,
		$$
		and
		$$
		\tilde{L}^{\rho}_T(B^s_{p,r})\hookrightarrow L^{\rho}_T(B^s_{p,r})  ~~~~\text{if}~~~~r\leq\rho.
		$$
		
		Let $u$ and $v$ be tempered distributions in $\mathscr{S}'$, then the nonhomogeneous $\rm paraproduct$ of $v$ by $u$ is defined as follows:
		\begin{align*}
			T_u v = \sum_{j} S_{j-1}u\Delta_j v,
		\end{align*}
		and the nonhomogeneous $\rm remainder$ of $u$ and $v$ is defined by
		\begin{align*}
			R(u,v) = \sum_{|k-j|\leq1} \Delta_{k} u \Delta_{j} v\triangleq\sum_{k\geq-1} \Delta_{k}u\tilde{\Delta}_k v.
		\end{align*}
		At least formally, we obtain the so-called Bony's decomposition:
		\begin{align*}
			uv=T_u v + T_v u + R(u,v).
		\end{align*}
	\begin{prop}\cite{Bahouri2011}\label{prop1}
		For any $(s,t)\in\mathbb{R}\times(-\infty,0)$ and $(p,p_1,p_2,r,r_1,r_2)\in[1,\infty]^6$, there exists a constant $C$ such that
		\begin{align*}
			\|T_u v\|_{B^s_{p,r}} \leq C^{1+|s|} \|u\|_{L^{p_1}}\|v\|_{B^s_{p_2,r}},
		\end{align*}
		with $(u,v)\in L^{p_1}\times B^s_{p_2,r}$ and $\frac{1}{p}=\frac{1}{p_1}+\frac{1}{p_2}$. Moreover, we have
		\begin{align*}
			\|T_u v\|_{B^{s+t}_{p,r}} \leq \frac{C^{1+|s+t|}}{-t} \|u\|_{B^{t}_{\infty,r_1}}\|v\|_{B^s_{p,r_2}},
		\end{align*}
		with $(u,v) \in B^{t}_{\infty,r_1}\times B^s_{p,r_2}$ and $\frac{1}{r}=\min\{1,\frac{1}{r_1}+\frac{1}{r_2}\}$.
	\end{prop}
	\begin{prop}\cite{Bahouri2011}\label{prop2}
		A constant $C$ exists which satisfies the following inequalities. Let $(s_1,s_2)\in\mathbb{R}^2$ and $(p_1,p_2,r_1,r_2) \in [1,\infty]^4$. Assume that $$\frac{1}{p}=\frac{1}{p_1}+\frac{1}{p_2}\leq1~~~~~\text{and}~~~~~\frac{1}{r}=\frac{1}{r_1}+\frac{1}{r_2}\leq1.$$
		If $s_1+s_2>0$, for any $(u,v)\in B^{s_1}_{p_1,r_1}\times B^{s_2}_{p_2,r_2}$, then we have
		\begin{align*}
			\|R(u,v)\|_{B^{s_1+s_2}_{p,r}} \leq \frac{C^{1+|s_1+s_2|}}{s_1+s_2} \|u\|_{B^{s_1}_{p_1,r_1}}\|v\|_{B^{s_2}_{p_2,r_2}}.
		\end{align*}
		If $r=1$ and $s_1+s_2=0$, for any $(u,v)\in B^{s_1}_{p_1,r_1}\times B^{s_2}_{p_2,r_2}$, then we have
		\begin{align*}
			\|R(u,v)\|_{B^{s_1+s_2}_{p,\infty}} \leq C^{1+|s_1+s_2|} \|u\|_{B^{s_1}_{p_1,r_1}}\|v\|_{B^{s_2}_{p_2,r_2}}.
		\end{align*}
	\end{prop}

	We introduce the following lemma to describe the action of the heat equation.
	\begin{lemm}\label{lemma1}\cite{Bahouri2011}
		Let $\mathscr{C}$ be an annulus. Positive constants $c$ and $C$ exist such that for any $p\in[1,+\infty]$ and any couple $(t,\lambda)$ of positive real numbers, we have
		\begin{align*}
			{\rm Supp}~\hat{u} \subset \lambda\mathscr{C} \Rightarrow \|e^{t\Delta}u\|_{L^p} \leq Ce^{-ct\lambda^2}\|u\|_{L^p}.
		\end{align*}
	\end{lemm}
	The following commutator lemma is useful to estimate $\Gamma$ for the case $p\in[2,\infty)$.
	\begin{lemm}\cite{2015Elgindi,2011Hmidi}\label{lemma2}
		Let ${\rm div}~u=0$ and $\mathscr{R}=\Delta^{-1}{\rm curl}~{\rm div}$. For any $(p,r)\in[2,\infty)\times[1,\infty]$, there exists a constant $C=C(p,r)$ such that
		\begin{align*}
			\|[\mathscr{R},u\cdot\nabla]\tau\|_{B^{0}_{p,r}}\leq C\|\nabla u\|_{L^p}\left(\|\tau\|_{B^0_{\infty,r}}+\|\tau\|_{L^p}\right).
		\end{align*}
	\end{lemm}
Denote $A(D) f=\mathscr{F}^{-1}(A(\xi) \widehat{f})$. In the next proposition, we give some properties of the Riesz operator.
	\begin{prop}\cite{2011Hmidi}\label{prop3}
		Let $\mathscr{R}_i$ be the Riesz operator $\mathscr{R}_i=\frac{\partial_i}{|D|}$. Then the following hold true.\\
		(1) For any $p\in(1,\infty)$, there exists a positive constant $C$ such that
		\begin{align*}
			\|\mathscr{R}_i\|_{\mathcal{L}(L^p)}\leq C.
		\end{align*}
		(2) Let $\chi\in\mathscr{D}(\mathbb{R}^d)$. Then, there exists a positive constant $C$ such that
		\begin{align*}
			\||D|^s\chi(2^{-q}|D|)\mathscr{R}_i\|_{\mathcal{L}(L^p)}\leq C2^{qs},
		\end{align*}
		~~~~~for any $(p,s,q)\in [1,\infty]\times(0,\infty)\times{\rm N}$.\\
		(3) Let $\mathscr{C}$ be a fixed ring. Then, there exists $\psi\in\mathscr{S}$ whose spectrum dose not meet the origin such that
		\begin{align*}
			\mathscr{R}_if = 2^{qd}\psi(2^q \dot) \ast f,
		\end{align*}
		~~~~~for any $f$ with Fourier transform supported in $2^q\mathscr{C}$.
	\end{prop}
	\begin{lemm}\cite{2011Hmidi}\label{lemma3}
		Given $(p,m)\in[1,\infty]^2$ such that $p\geq m'$ with $m'$ the conjugate exponent of $m$. Let $f$, $g$ and $h$ be three functions such that $\nabla f\in L^p$, $g\in L^m$ and $xh\in L^{m'}$. Then,
		\begin{align*}
			\|h\ast(fg)-f(h\ast g)\|_{L^p} \leq \|xh\|_{L^{m'}}\|\nabla f\|_{L^p}\|g\|_{L^m} .
		\end{align*}
	\end{lemm}
	As explained in the introduction, we give the proof of the commutator lemma between the Riesz operator $\mathscr{R}_i$ and the convection operator $u\cdot\nabla$ for the case $p\in(1,2)$.
	\begin{prop}\label{prop4}
        For any $p\in(1,2)$, there exists a constant $C=C(p,r)$ such that
		\begin{align*}
			\|[\mathscr{R}_i,u\cdot\nabla]\tau\|_{L^p} \leq C\|\nabla u\|_{L^p}\left(\|\tau\|_{L^q}+\|\tau\|_{B^{0}_{\infty,1}}\right),
		\end{align*}
		where $\frac{p}{p-1}\leq q<\infty$.
	\end{prop}
	\begin{proof}
		By Bony's decomposition, we spilt the commutator into three parts,
		\begin{align}\label{2eq1}
			[\mathscr{R}_i,u\cdot\nabla]\tau &= \sum_{j\in N} [\mathscr{R}_i,S_{q-1}u\cdot\nabla]\Delta_j\tau + \sum_{j\in N} [\mathscr{R}_i,\Delta_{j}u\cdot\nabla]S_{j-1}\tau + \sum_{j\geq-1} [\mathscr{R}_i,\Delta_{j}u\cdot\nabla]\tilde{\Delta}_j \tau \\ \notag
			&\triangleq\sum_{j\in N} B^1_j + \sum_{j\in N} B^2_j + \sum_{j\geq-1} B^3_j .
		\end{align}
		Let's start with the estimate of $B^1_j$ in \eqref{2eq1}. By Proposition \ref{prop3}, there exists $h \in \mathscr{S}$ whose spectrum does not meet the origin such that
		\begin{align*}
			B^1_j = h_j \ast \left(S_{j-1}u\cdot\nabla\Delta_j\tau\right) - S_{j-1}u\cdot  \left(h_j \ast\nabla\Delta_j\tau\right),
		\end{align*}
		where $h_j(x) = 2^{dq} h(2^qx)$. According to Bernstein's inequality, $\|xh_j\|_{L^1} = 2^{-j}\|xh\|_{L^1}$ and Lemma \ref{lemma3} with $m=\infty$, we obtain
		\begin{align*}
			\|B^1_j\|_{L^p} &\lesssim \|xh_j\|_{L^1} \|\nabla S_{j-1}u\|_{L^p}\|\nabla\Delta_j\tau\|_{L^{\infty}} \\ \notag
			&\lesssim \|\nabla u\|_{L^p}\|\Delta_j\tau\|_{L^\infty}.
		\end{align*}
		Then, we infer that
		\begin{align}
			 \|\sum_{j\in N}B^1_j\|_{L^p} \lesssim \sum_{j\in N} \|B^1_j\|_{L^p} \lesssim  \|\nabla u\|_{L^p}\|\tau\|_{B^0_{\infty,1}}.
		\end{align}
		Similarly, we can write
		\begin{align*}
			B^2_j = h_j \ast \left(\Delta_ju\cdot\nabla S_{j-1}\tau\right) - \Delta_ju\cdot  \left(h_j \ast\nabla S_{j-1}\tau\right).
		\end{align*}
		Applying convolution inequality, we deduce that
		\begin{align}
			\|\sum_{j\in N}B^2_j\|_{L^p} &\lesssim \sum_{j\in N}2^{-j}\|\Delta_j\nabla u\|_{L^p}\|S_{j-1}\nabla\tau\|_{L^{\infty}} \\ \notag
			&\lesssim \|\nabla u\|_{L^p}\sum_{j\in N}\sum_{q\leq j-2}2^{q-j}\|\Delta_q\tau\|_{L^{\infty}} \\ \notag
		    &\lesssim  \|\nabla u\|_{L^p}\|\tau\|_{B^0_{\infty,1}}.
		\end{align}
		Since ${\rm div}~u=0$, we rewrite $B^3_j$ as
		\begin{align}\label{2eq2}
			\sum_{j\geq-1} B^3_j &= \sum_{j\geq2} \mathscr{R}_i{\rm div}~(\Delta_ju\tilde{\Delta}_j\tau)-\sum_{j\geq2} {\rm div}~(\Delta_ju\mathscr{R}_i\tilde{\Delta}_j\tau) + \sum_{j\leq1} [\mathscr{R}_i,\Delta_ju\cdot\nabla]\tilde{\Delta}_j\tau\\ \notag
			&\triangleq\sum_{j\geq2} D^1_j+\sum_{j\geq2} D^2_j + \sum_{j\leq1} D^3_j.
		\end{align}
		Using \eqref{2eq2} and Proposition \ref{prop3} , one can arrive at
		\begin{align*}
			\|D^1_j\|_{L^p} \lesssim 2^j\|\Delta_ju\|_{L^p}\|\tilde{\Delta}_j\tau\|_{L^\infty}.
		\end{align*}
		For $j\geq2$, $\tilde{\Delta}_j\tau$ is supported away from zero. Then, there exists $k\in\mathscr{S}$ such that
		$\mathscr{R}_i\tilde{\Delta}_j\tau=k\ast\tilde{\Delta}_j\tau$. Applying Young's inequality, then we have
		\begin{align*}
			\|D^2_j\|_{L^p} &\lesssim 2^j\|\Delta_ju\|_{L^p}\|\mathscr{R}_i\tilde{\Delta}_j\tau\|_{L^\infty}\\ \notag
			&\lesssim 2^j\|\Delta_ju\|_{L^p}\|\tilde{\Delta}_j\tau\|_{L^\infty}.
		\end{align*}
		Thus we conclude that
		\begin{align} \label{ineq1}
			\|\sum_{j\geq2}(D^1_j+D^2_j)\|_{L^p}\lesssim \|\nabla u\|_{L^p}\|\tau\|_{B^0_{\infty,1}}.
		\end{align}
		For $D^3_j$ in \eqref{2eq2}, see \cite{2011Hmidi}, there exists $\chi_1\in\mathscr{D}(\mathbb{R}^d)$ and convolution kernel $g$ satisfying $|g(x)|\lesssim (1+|x|)^{-d-1}$ such that
		\begin{align*}
			D^3_j &=  [\mathscr{R}_i\partial_k,\Delta_{j}u^k]\tilde{\Delta}_j \tau \\ \notag
			& =  [\partial_k\chi_1(D)\mathscr{R}_i,\Delta_{j}u^k]\tilde{\Delta}_j \tau \\ \notag
			& = g\ast\left(\Delta_{j}u\cdot\tilde{\Delta}_j \tau\right)-\Delta_{j}u\cdot\left(g\ast\tilde{\Delta}_j \tau\right).
		\end{align*}
		One can easily deduce that $xg\in L^m$ for any $m>1$. By Lemma \ref{lemma3}, we obtain
		\begin{align} \label{ineq2}
			\|\sum_{j\leq 1}D^3_j\|_{L^p} &\lesssim \sum_{j\leq 1}2^{-j}\|xg\|_{L^{q'}}\|\Delta_j\nabla u\|_{L^p}\|\tilde{\Delta}_j\tau\|_{L^q} \\ \notag
			&\lesssim \|\nabla u\|_{L^p}\|\tau\|_{L^q},
		\end{align}
		provided $1<q'=\frac{q}{q-1}\leq p$. Combining \eqref{ineq1} and \eqref{ineq2}, we infer that
		\begin{align}
			 \|\sum_{j\geq-1}B^3_j\|_{L^p} &\lesssim  \|\nabla u\|_{L^p}\left(\|\tau\|_{L^q}+\|\tau\|_{B^0_{\infty,1}}\right).
		\end{align}
		We thus complete the proof of Proposition \ref{prop4}.
	\end{proof}
	\begin{coro}\label{coro1}
		Let ${\rm div}~u=0$ and $\mathscr{R}=\Delta^{-1}{\rm curl}~{\rm div}$. Then for any $p\in(1,2)$, there exists a constant $C=C(p,r)$ such that
		\begin{align*}
			\|[\mathscr{R},u\cdot\nabla]\tau\|_{L^p}\leq C\|\nabla u\|_{L^p}\left(\|\tau\|_{L^q}+\|\tau\|_{B^0_{\infty,1}}\right).
		\end{align*}
	\end{coro}
	\begin{proof}
		Firstly, the commutator can be rewritten as follows:
		\begin{align}
			[\mathscr{R},u\cdot\nabla] \tau &= \mathscr{R}_i\mathscr{R}_j(u\cdot\nabla\tau) - u\cdot\nabla\left(\mathscr{R}_i\mathscr{R}_j(\tau)\right) \\ \notag
			&=\mathscr{R}_i\left([\mathscr{R}_j,u\cdot\nabla]\tau\right) + [\mathscr{R}_i,u\cdot\nabla](\mathscr{R}_j\tau).
		\end{align}
		By virtue of Propositions \ref{prop3} and \ref{prop4}, we infer that
		\begin{align*}
			\|\mathscr{R}_i\left([\mathscr{R}_j,u\cdot\nabla]\tau\right)\|_{L^p} &\lesssim \|[\mathscr{R}_j,u\cdot\nabla]\tau\|_{L^p} \\ \notag
			&\lesssim \|\nabla u\|_{L^p}\left(\|\tau\|_{L^q}+\|\tau\|_{B^0_{\infty,1}}\right).
		\end{align*}
		Similarly, we deduce that
		\begin{align*}
			\|[\mathscr{R}_i,u\cdot\nabla](\mathscr{R}_j\tau_k)\|_{L^p}
			&\lesssim \|\nabla u\|_{L^p}\left(\|\mathscr{R}_j\tau\|_{L^q}+\|\Delta_{-1}\mathscr{R}_j\tau\|_{B^0_{\infty,1}}+\|(Id-\Delta_{-1})\mathscr{R}_j\tau\|_{B^0_{\infty,1}}\right) \\ \notag
			&\lesssim \|\nabla u\|_{L^p}\left(\|\mathscr{R}_j\tau\|_{L^q}+\|\mathscr{R}_j\tau\|_{B^0_{\infty,1}}\right) \\ \notag
			&\lesssim  \|\nabla u\|_{L^p}\left(\|\tau\|_{L^q}+\|\tau\|_{B^0_{\infty,1}}\right).
		\end{align*}
		Thus we conclude that
		\begin{align}
			\|[\mathscr{R},u\cdot\nabla]\tau\|_{L^p}\lesssim \|\nabla u\|_{L^p}\left(\|\tau\|_{L^q}+\|\tau\|_{B^0_{\infty,1}}\right).
		\end{align}
		This completes the proof of Corollary \ref{coro1}.
	\end{proof}
	\section{Global estimates.}
	\par
	In this section, we are going to establish global estimates with low regularity, which are extremely significant to global existence of weak solutions. \subsection{The co-rotation case}
	For convenience, we take $d=2$ and $a=\mu=1$ in \eqref{eq2} in this subsection.
	\subsubsection{Basic energy estimates}
	\begin{lemm}\cite{DLY1}\label{3lemma1}
		Let $(u,\tau)$ be a smooth solution to \eqref{eq2}. Suppose that $\tau_0\in L^p$ with $p\in[1,\infty]$, for any $t>0$, then we have
		\begin{align*}
			\|\tau\|_{L^p} \leq e^{-t}\|\tau_0\|_{L^p}.
		\end{align*}
		If $(u_0,\tau_0)\in L^2$, then we obtain
		\begin{align*}
		\|u\|_{L^2} \leq  \|u_0\|_{L^2} + \|\tau_0\|_{L^2},~~~~e^{2t}\|\tau\|^2_{L^2} + 2\int_0^t e^{2s}\|\nabla\tau\|^2_{L^2} ds= \|\tau_0\|^2_{L^2}.
		\end{align*}
	\end{lemm}
	\begin{prop}\cite{DLY1}\label{3lemma2}
		Suppose $(u,\tau)$ is a smooth solution to \eqref{eq2}  with initial data $u_0\in H^1$ and $\tau_0\in H^1\cap L^{\infty}$, then there exists a positive constant $C$ such that for any $t>0$,
		\begin{align*}
			\|(u,\tau)\|_{H^1} \leq \|(u_0,\tau_0)\|^2_{H^1} e^{C\left(1+\|\tau_0\|^2_{L^{\infty}}+\|\tau_0\|^2_{L^2}\right)}.
		\end{align*}
	\end{prop}
	\subsubsection{Global estimates of vorticity}
	We now establish global estimates of vorticity for the 2-D co-rotation inviscid Oldroyd-B model with some large data. Under different integrability conditions, we divide it into two cases to deal with $\|\Omega\|_{L^p}$, which is equivalent to $\|\nabla u\|_{L^p}$ for $1<p<\infty$. The first case $p\in(1,2)$ is given by the following proposition.
	\begin{prop}\label{case1}
		Let $p\in(1,2)$. Suppose $(u,\tau)$ is a smooth solution to system \eqref{eq2} with initial data satisfying the condition $(1)$ in Theorem \eqref{theo1}, then there exists some sufficiently small constant $c_1$, which is not dependent on the initial data, such that for any $t>0$,
		\begin{align*}
			\int_0^t \|\tau\|_{B^0_{\infty,1}} ds \leq c_1
			~~~~~\text{and} ~~~~~\|\nabla u\|_{L^p} \leq C\left(\|\nabla u_0\|_{L^p}+\|\tau_0\|_{L^p}\right).
		\end{align*}
		Moreover, if additionally $u_0\in \dot{W}^{1,r}$ and $\tau_0\in L^{r}$ with $r\in[2,\infty)$, one can arrive at
		\begin{align*}
			\|\nabla u\|_{L^r} &\leq C\left(\|\nabla u_0\|_{L^r} + \|\tau_0\|_{L^r}\right).
		\end{align*}
	\end{prop}
	\begin{proof}
		To begin with, we cancel ${\rm div}~\tau$ and $\Delta\tau$ by virtue of the structural trick
		\begin{align}\label{1ineq1}
		\Gamma = \Omega-\mathscr{R}\tau,
		\end{align}
		where $\mathscr{R} = \Delta^{-1}{\rm curl}~{\rm div}$. We deduce from \eqref{1ineq1} and \eqref{eq2} that
		\begin{align}\label{1ineq2}
			\partial_t\Gamma + u\cdot\nabla\Gamma =\mathscr{R}\tau + \mathscr{R}Q(\Omega,\tau) + [\mathscr{R},u\cdot\nabla]\tau.
		\end{align}
		Taking the $L^p$ norm to \eqref{1ineq2}, then we infer that
		\begin{align}\label{1ineq3}
			\frac{d}{dt}\|\Gamma\|_{L^p} \lesssim \|\mathscr{R}\tau\|_{L^p} + \|\mathscr{R}Q(\Omega,\tau)\|_{L^p} + \|[\mathscr{R},u\cdot\nabla]\tau\|_{L^p}.
		\end{align}
		By virtue of \eqref{1ineq3}, Proposition \ref{prop3} and Corollary \ref{coro1} with $q\in[\frac{p}{p-1},\infty)$, we obtain
		\begin{align}\label{1ineq4}
			\frac{d}{dt}\|\Gamma\|_{L^p} &\lesssim \|\tau\|_{L^p} + \|\Omega\|_{L^p}\|\tau\|_{L^{\infty}} + \|\nabla u\|_{L^p}\left(\|\tau\|_{L^q}
			+\|\tau\|_{B^0_{\infty,1}}\right)\\ \notag
			&\lesssim \|\tau\|_{L^p} + \|\Omega\|_{L^p}\left(\|\tau\|_{H^1}
			+\|\tau\|_{B^0_{\infty,1}}\right)\\ \notag
			&\lesssim  \|\Gamma\|_{L^p}\left(\|\tau\|_{H^1}+\|\tau\|_{B^0_{\infty,1}}\right) + \|\tau\|_{L^p}\left(1+\|\tau\|_{H^1}+\|\tau\|_{B^0_{\infty,1}}\right).
		\end{align}
		Using Lemma \ref{3lemma1}, we have $$\int_0^t\|\tau\|_{H^1}ds\lesssim\int_0^t\|\tau\|_{L^2}ds+(\int_0^te^{-2s}ds)^{\frac 1 2}(\int_0^te^{2s}\|\nabla\tau\|^2_{L^2}ds)^{\frac 1 2}\lesssim \|\tau_0\|_{L^2}.$$
		Integrating \eqref{1ineq4} over $[0,t]$, we infer the above estimate and Lemma \ref{3lemma1} that
		\begin{align}\label{1ineq5}
			\|\Gamma\|_{L^\infty_t(L^p)} &\lesssim \|\Gamma_0\|_{L^p} + \|\tau_0\|_{L^p}+\|\tau_0\|_{L^p}\|\tau_0\|_{L^2} + \|\tau_0\|_{L^2}\|\Gamma\|_{L^\infty_t(L^p)}\\ \notag
			&~~~+\|\tau_0\|_{L^p}\int_0^t\|\tau\|_{B^0_{\infty,1}}ds+\|\Gamma\|_{L^\infty_t(L^p)}\int_0^t\|\tau\|_{B^0_{\infty,1}}ds.
		\end{align}
		
		We now focus on the estimate of $\int_0^t\|\tau\|_{B^0_{\infty,1}}ds$. Taking $\Delta_j$ with $j\geq 0$ to system $(\ref{eq2})_2$ and applying Duhamel's principle, we deduce that
		\begin{align}\label{1ineq6}
			\Delta_j\tau = e^{\Delta t}e^{-t}\Delta_j\tau_0 + \int_0^t e^{\Delta(t-s)}e^{-(t-s)}\Delta_j(Q(\Omega,\tau)+u\cdot\nabla\tau)ds.
		\end{align}
		Taking $L^{\infty}$ norm to \eqref{1ineq6} and integrating over $[0,t]$, for $j\geq 0$, we infer from Lemma \ref{lemma1} that
		\begin{align}\label{1ineq7}
			\int_0^t\|\Delta_j\tau\|_{L^{\infty}}ds &\lesssim \int_0^t e^{-2^{2j}s}\|\Delta_j\tau_0\|_{L^{\infty}}ds + \int_0^t\int_0^s e^{-2^{2j}(s-s')}\|\Delta_jQ(\Omega,\tau)\|_{L^{\infty}}ds'ds \\ \notag
			&~~~+\int_0^t\int_0^s e^{-2^{2j}(s-s')}\|\Delta_j(u\cdot\nabla\tau)\|_{L^{\infty}}ds'ds.
		\end{align}
		Applying Bernstain's inequality, we infer that
		\begin{align}\label{1ineq8}
			\int_0^t e^{-2^{2j}s}\|\Delta_j\tau_0\|_{L^{\infty}}ds \lesssim \int_0^t2^{j}e^{-2^{2j}s}\|\Delta_j\tau_0\|_{L^{2}}ds
			\lesssim 2^{-j}\|\tau_0\|_{L^{2}}.
		\end{align}
		We infer from Young's inequality that
		\begin{align}\label{1ineq9}
			\int_0^t\int_0^s e^{-2^{2j}(s-s')}\|\Delta_jQ(\Omega,\tau)\|_{L^{\infty}}ds'ds &\lesssim \int_0^t\int_0^s 2^{2j}e^{-2^{2j}(s-s')}c_j\|Q(\Omega,\tau)\|_{B^{-2}_{\infty,1}}ds'ds \\ \notag
			&\lesssim  c_j\int_0^t\|Q(\Omega,\tau)\|_{B^{-2}_{\infty,1}}ds,
		\end{align}
		and
		\begin{align}\label{1ineq10}
			\int_0^t\int_0^s e^{-2^{2j}(s-s')}\|\Delta_j(u\cdot\nabla\tau)\|_{L^{\infty}}ds'ds &\lesssim 2^{(-2+\frac{2}{p})j}\int_0^t\int_0^s 2^{2j}e^{-2^{2j}(s-s')}\|u\cdot\nabla \tau\|_{L^p}ds'ds \\ \notag
			&\lesssim  2^{(-2+\frac{2}{p})j}\int_0^t\|u\cdot\nabla \tau\|_{L^p}ds ,
		\end{align}
		where $\{c_j\}_{j\geq-1}\in l^1$. Combining the estimates \eqref{1ineq7}-\eqref{1ineq10}, we deduce that
		\begin{align}\label{1ineq11}
			\int_0^t\|\tau\|_{B^0_{\infty,1}}ds &\lesssim \int_0^t\|\Delta_{-1}\tau\|_{L^{\infty}}ds+\sum_{j\geq 0}\int_0^t\|\Delta_j\tau\|_{L^{\infty}}ds \\ \notag
			&\lesssim \|\tau_0\|_{L^{2}} + \int_0^t\|Q(\Omega,\tau)\|_{B^{-2}_{\infty,1}}ds + \int_0^t\|u\cdot\nabla \tau\|_{L^p}ds.
		\end{align}
		Let's deal with the nonlinear terms in \eqref{1ineq11}.
		Taking Bony's decomposition, we get
		\begin{align}\label{1ineq12}
			\int_0^t\|Q(\Omega,\tau)\|_{B^{-2}_{\infty,1}}ds &\lesssim
			\int_0^t \|T_{\Omega}\tau\|_{B^{-2}_{\infty,1}}+\|T_{\tau}\Omega\|_{B^{-2}_{\infty,1}}+\|R(\Omega,\tau)\|_{B^{-2}_{\infty,1}}ds.
		\end{align}
		Applying Proposition \ref{prop1} with $p\in(1,\infty)$, one can arrive at
		\begin{align}\label{1ineq13}
			\int_0^t \|T_{\Omega}\tau\|_{B^{-2}_{\infty,1}}+\|T_{\tau}\Omega\|_{B^{-2}_{\infty,1}} ds &\lesssim
			\int_0^t \|\tau\|_{B^{-2+\frac{2}{p}}_{\infty,1}}\|\Omega\|_{B^{-\frac{2}{p}}_{\infty,\infty}}ds \\ \notag
			&\lesssim \int_0^t \|\tau\|_{H^{1}}\|\Omega\|_{L^p}ds \\ \notag
			&\lesssim\int_0^t \|\tau\|_{H^{1}}\|\tau\|_{L^p} + \|\tau\|_{H^{1}}\|\Gamma\|_{L^p}ds \\ \notag
			&\lesssim \|\tau_0\|_{L^{2}}\|\tau_0\|_{L^p} + \|\tau_0\|_{L^2}\|\Gamma\|_{L^\infty_t(L^p)}.
		\end{align}
		Moreover, by virtue of Proposition \ref{prop2} with $p\in(1,2]$, we infer that
		\begin{align} \label{1ineq14}
			\int_0^t \|R(\Omega,\tau)\|_{B^{-2}_{\infty,1}} ds  &\lesssim \int_0^t \|R(\Omega,\tau)\|_{B^{\varepsilon}_{1,\infty}} ds \\ \notag
			&\lesssim \int_0^t \|\Omega\|_{B^{0}_{p,\infty}}\|\tau\|_{B^{\varepsilon}_{p',\infty}} ds \\ \notag
			&\lesssim \int_0^t \|\Omega\|_{B^{0}_{p,\infty}}\|\tau\|_{B^{\varepsilon-1+\frac{2}{p}}_{2,\infty}} ds \\ \notag
			&\lesssim \int_0^t \|\Omega\|_{L^p}\|\tau\|_{H^1} ds \\ \notag
			&\lesssim\left\|\tau_0\|_{L^2}\|\Gamma\|_{L^\infty_t(L^p)} + \|\tau_0\|_{L^2}\|\tau_0\|_{L^p},\right.
		\end{align}
		where $\varepsilon=2-\frac{2}{p}>0$.
		Then, we deduce from \eqref{1ineq12}-\eqref{1ineq14} that
		\begin{align}\label{1ineq15}
			\int_0^t\|Q(\Omega,\tau)\|_{B^{-2}_{\infty,1}}ds\lesssim \|\tau_0\|_{L^2}\|\Gamma\|_{L^\infty_t(L^p)}+ \|\tau_0\|_{L^2}\|\tau_0\|_{L^p}.
		\end{align}
		For $p\in(1,2)$, we infer from Proposition \ref{prop3} and Lemma \ref{3lemma1} that
		\begin{align}\label{1ineq16}
			\int_0^t\|u\cdot\nabla \tau\|_{L^p}ds
			&\lesssim
			\int_0^t\|u\|_{L^{\frac{2p}{2-p}}}\|\nabla\tau\|_{L^2}ds  \\ \notag
			&\lesssim \int_0^t\|\nabla u\|_{L^{p}}\|\nabla\tau\|_{L^2}ds \\ \notag
			&\lesssim \int_0^t \|\tau\|_{L^{p}}\|\nabla\tau\|_{L^2} + \|\Gamma\|_{L^{p}}\|\nabla\tau\|_{L^2} ds \\ \notag
			&\lesssim \|\tau_0\|_{L^2}\|\tau_0\|_{L^{p}} + \|\tau_0\|_{L^2}\|\Gamma\|_{L^\infty_t(L^p)}.
		\end{align}
		Combining \eqref{1ineq11}, \eqref{1ineq15} and \eqref{1ineq16}, we conclude that
		\begin{align}\label{1ineq17}
			\int_0^t\|\tau\|_{B^0_{\infty,1}}ds \lesssim \|\tau_0\|_{L^2} + \|\tau_0\|_{L^2}\|\tau_0\|_{L^{p}} + \|\tau_0\|_{L^2}\|\Gamma\|_{L^\infty_t(L^p)}.
		\end{align}
		Assume that $\int_0^t\|\tau\|_{B^0_{\infty,1}}ds\leq c_1$ with $c_1\leq \frac {1} {4C}$, we infer from \eqref{1ineq5} and \eqref{con1} with $c\leq \frac {1} {4C}$ that
		\begin{align}\label{1ineq18}
			\|\Gamma\|_{L^\infty_t(L^p)} &\leq C\|\Gamma_0\|_{L^p}
			+ C\|\tau_0\|_{L^p}.
		\end{align}
		Plugging \eqref{1ineq18} into \eqref{1ineq17}, we deduce that
		\begin{align*}
			\int_0^t\|\tau\|_{B^0_{\infty,1}}ds &\leq C\|\tau_0\|_{L^2} + C\|\tau_0\|_{L^2}\|\tau_0\|_{L^{p}} + C\|\tau_0\|_{L^2}\|\Gamma_0\|_{L^p}  \\ \notag
			&\leq C\|\tau_0\|_{L^2} + C\|\tau_0\|_{L^2}\|\tau_0\|_{L^{p}} + C\|\tau_0\|_{L^2}\|\nabla u_0\|_{L^p} \\ \notag
			& \leq \frac{c_1}{2} ,
		\end{align*}
		where we choose $c \in (0,\frac{c_1}{2C}]$ in \eqref{con1}. By continuity argument,  for any $t>0$, we conclude that
		\begin{align}\label{1ineq19}
		\int_0^t\|\tau\|_{B^0_{\infty,1}}ds\leq c_1.
		\end{align}
		
		Moreover,
		if $u_0\in \dot{W}^{1,r}$ and $\tau_0\in L^{r}$ with $r\in[2,\infty)$, one can deduce from Lemma \ref{lemma2} that
		\begin{align}\label{1ineq20}
			\frac{d}{dt}\|\Gamma\|_{L^r} &\lesssim
			\|\tau\|_{L^r} + \|\Omega\|_{L^r}\|\tau\|_{L^{\infty}} + \|\nabla u\|_{L^r}\left(\|\tau\|_{L^r}+\|\tau\|_{B^0_{\infty,1}}\right) \\ \notag
			&\lesssim  \|\tau\|_{L^r}\left(1+\|\tau\|_{H^1}+\|\tau\|_{B^0_{\infty,1}}\right)  + \|\Gamma\|_{L^r}\left(\|\tau\|_{H^1}+\|\tau\|_{B^0_{\infty,1}}\right).
		\end{align}
		Integrating \eqref{1ineq20} over $[0,t]$, we obtain
		\begin{align*}
			\|\Gamma\|_{L^\infty_t(L^r)} &\lesssim \|\Gamma_0\|_{L^r} + \|\tau_0\|_{L^r}+\|\tau_0\|_{L^r}\|\tau_0\|_{L^2} + \|\tau_0\|_{L^2}\|\Gamma\|_{L^\infty_t(L^r)}\\ \notag
			&~~~+\|\tau_0\|_{L^r}\int_0^t\|\tau\|_{B^0_{\infty,1}}ds+\|\Gamma\|_{L^\infty_t(L^r)}\int_0^t\|\tau\|_{B^0_{\infty,1}}ds.
		\end{align*}
		Using \eqref{con1} and \eqref{1ineq19}, we deduce that
		\begin{align*}
			\|\nabla u\|_{L^r}&\lesssim \|\Gamma\|_{L^r} + \|\tau_0\|_{L^r} \\ \notag
			&\lesssim \|\Gamma_0\|_{L^r} + \|\tau_0\|_{L^r} \\ \notag
			&\lesssim \|\nabla u_0\|_{L^r} + \|\tau_0\|_{L^r}.
		\end{align*}
		We thus complete the proof of Proposition \ref{case1}.
	\end{proof}
	Then, we deal with the second case $p\in[2,\infty)$ in the following proposition.
	\begin{prop}\label{case2}
		Let $p\in[2,\infty)$. Suppose $(u,\tau)$ is a smooth solution to system \eqref{eq2} with initial data satisfying the condition $(2)$ in Theorem \eqref{theo1}, then there exists some sufficiently small constant $c_1$, which is not dependent on the initial data, such that for any $t>0$,
		\begin{align*}
			\int_0^t \|\tau\|_{B^0_{\infty,1}} ds \leq c_1
			~~~~~\text{and} ~~~~~\|\nabla u\|_{L^p} \leq C(\|\nabla u_0\|_{L^p} + \|\tau_0\|_{L^p}).
		\end{align*}
		Moreover, if $p\in(2,\infty)$, one can arrive at
		\begin{align*}
			\|\tau\|_{\tilde{L}^1(0,\infty;B^{2-\frac{2}{p}}_{\infty,\infty})} \leq C\left(\|u_0\|_{L^2\cap \dot{W}^{1,p}} + \|\tau_0\|_{L^2\cap L^p}\right).
		\end{align*}
	\end{prop}
	\begin{proof}
		To start with, we derive the energy estimate of $\|\Gamma\|_{L^p}$. Taking $L^p$ norm to \eqref{1ineq2}, we have
		\begin{align*}
			\frac{d}{dt}\|\Gamma\|_{L^p} \lesssim \|\mathscr{R}\tau\|_{L^p} + \|\mathscr{R}Q(\Omega,\tau)\|_{L^p} + \|[\mathscr{R},u\cdot\nabla]\tau\|_{L^p}.
		\end{align*}
		By virtue of Proposition \ref{prop3} and Lemma \ref{lemma2}, we infer that
		\begin{align}\label{2ineq1}
			\frac{d}{dt}\|\Gamma\|_{L^p} &\lesssim  \|\tau\|_{L^p}+\|\Omega\|_{L^p}\left(\|\tau\|_{L^p}+\|\tau\|_{B^0_{\infty,1}}\right) \\ \notag
			&\lesssim  \|\Gamma\|_{L^p}\left(\|\tau\|_{L^p}+\|\tau\|_{B^0_{\infty,1}}\right) + \|\tau\|_{L^p}\left(1+\|\tau\|_{L^p}+\|\tau\|_{B^0_{\infty,1}}\right) \\ \notag
			&\lesssim  \|\Gamma\|_{L^p}\left(\|\tau\|_{H^1}+\|\tau\|_{B^0_{\infty,1}}\right) + \|\tau\|_{L^p}\left(1+\|\tau\|_{H^1}+\|\tau\|_{B^0_{\infty,1}}\right).
		\end{align}
		Integrating \eqref{2ineq1} over $[0,t]$, we obtain
		\begin{align}\label{2ineq2}
			\|\Gamma\|_{L_t^\infty(L^p)} &\lesssim \|\Gamma_0\|_{L^p} + \|\tau_0\|_{L^p}+\|\tau_0\|_{L^2}\|\tau_0\|_{L^p} + \|\tau_0\|_{L^2}\|\Gamma\|_{L_t^\infty(L^p)}\\ \notag
			&~~~+\|\tau_0\|_{L^p}\int_0^t\|\tau\|_{B^0_{\infty,1}}ds+\|\Gamma\|_{L_t^\infty(L^p)}\int_0^t\|\tau\|_{B^0_{\infty,1}}ds.
		\end{align}

		Then we focus on the estimate of $\int_0^t\|\tau\|_{B^0_{\infty,1}}ds$. Analogously, we infer from \eqref{1ineq7} and Bernstain's inequality that
		\begin{align}\label{2ineq3}
			\int_0^t\|\tau\|_{B^0_{\infty,1}}ds &\lesssim \int_0^t\|\Delta_{-1}\tau\|_{L^{\infty}}ds+\sum_{j\geq 0}\int_0^t\|\Delta_j\tau\|_{L^{\infty}}ds \\ \notag
			&\lesssim\|\tau_0\|_{L^{2}} + \int_0^t \left(\|Q(\Omega,\tau)\|_{B^{-2}_{\infty,1}}+\|u\cdot\nabla\tau\|_{B^{-2}_{\infty,1}}\right)ds.
		\end{align}
		By virtue of Proposition \ref{prop2} with $p\in(2,\infty)$, we infer that
		\begin{align}\label{2ineq4}
			\int_0^t \|R(\Omega,\tau)\|_{B^{-2}_{\infty,1}} ds  &\lesssim \int_0^t \|R(\Omega,\tau)\|_{B^{0}_{\frac{2p}{p+2},\infty}} ds \\ \notag
			&\lesssim \int_0^t \|\Omega\|_{B^{0}_{p,\infty}}\|\tau\|_{B^{0}_{2,1}} ds\\ \notag
			&\lesssim \int_0^t \|\Gamma\|_{L^p}\|\tau\|_{H^1} +  \|\tau\|_{L^p}\|\tau\|_{H^1}ds\\ \notag
			&\lesssim \|\tau_0\|_{L^2}\|\Gamma\|_{L_t^\infty(L^p)} +\|\tau_0\|_{L^2}\|\tau_0\|_{L^p}.
		\end{align}
		Combining \eqref{1ineq13}, \eqref{1ineq14} and \eqref{2ineq4}, we deduce that
		\begin{align}\label{2ineq5}
			\int_0^t \|Q(\Omega,\tau)\|_{B^{-2}_{\infty,1}}ds
			&\lesssim \|\tau_0\|_{L^2}\|\Gamma\|_{L_t^\infty(L^p)} +\|\tau_0\|_{L^2}\|\tau_0\|_{L^p}.
		\end{align}
		By Bony's decomposition, we get
		\begin{align}\label{2ineq6}
			\int_0^t\|u\cdot\nabla\tau\|_{B^{-2}_{\infty,1}}ds &\lesssim
			\int_0^t \|T_{u}\tau\|_{B^{-1}_{\infty,1}}+\|T_{\tau}u\|_{B^{-1}_{\infty,1}}+\|R(u,\tau)\|_{B^{1}_{1,1}}ds.
		\end{align}
		Applying Propositions \ref{prop1} and \ref{prop2} leads to
		\begin{align}\label{2ineq7}
		\int_0^t \|T_{u}\tau\|_{B^{-1}_{\infty,1}} +\|R(u,\tau)\|_{B^{1}_{1,1}} ds &\lesssim \int_0^t \|u\|_{B^{-1}_{\infty,2}}\|\tau\|_{B^{0}_{\infty,2}}+\|u\|_{B^{0}_{2,2}}\|\tau\|_{B^{1}_{2,2}}ds \\ \notag
		&\lesssim \int_0^t \|u\|_{L^2}\|\tau\|_{H^1}ds \\ \notag
		&\lesssim \|\tau_0\|_{L^2}\left(\|u_0\|_{L^2} + \|\tau_0\|_{L^2}\right).
		\end{align}
		If $p=2$, then we deduce from Proposition \ref{prop1} that
		\begin{align}\label{2ineq8}
		\int_0^t \|T_{\tau}u\|_{B^{-1}_{\infty,1}} &\lesssim \int_0^t \|\tau\|_{B^{-1}_{\infty,2}}\|u\|_{B^{0}_{\infty,2}}ds \\ \notag
		&\lesssim \int_0^t \|u\|_{H^1}\|\tau\|_{L^2}ds \\ \notag
		&\lesssim   \|\tau_0\|^2_{L^2} + \|\tau_0\|_{L^2}\|\Gamma\|_{L_t^\infty(L^2)}\\ \notag
		&~~~+\|\tau_0\|_{L^2}\left(\|u_0\|_{L^2} + \|\tau_0\|_{L^2}\right).
		\end{align}
		Applying Proposition \ref{prop1} with $p\in(2,\infty)$, one can arrive at
		\begin{align}\label{2ineq9}
			\int_0^t \|T_{\tau}u\|_{B^{-1}_{\infty,1}} &\lesssim \int_0^t \|\tau\|_{B^{-1+\frac{2}{p}}_{\infty,1}}\|u\|_{B^{-\frac{2}{p}}_{\infty,\infty}}ds \\ \notag
			&\lesssim \int_0^t (\|u\|_{L^2}+\|\nabla u\|_{L^p})\|\tau\|_{H^1}ds \\ \notag
			&\lesssim   \|\tau_0\|_{L^2}\|\tau_0\|_{L^p} + \|\tau_0\|_{L^2}\|\Gamma\|_{L_t^\infty(L^p)}\\ \notag
			&~~~+\|\tau_0\|_{L^2}\left(\|u_0\|_{L^2} + \|\tau_0\|_{L^2}\right).
		\end{align}
		Thus, we deduce from \eqref{2ineq6}-\eqref{2ineq9} that
		\begin{align}\label{2ineq10}
			\int_0^t\|u\cdot\nabla\tau\|_{B^{-2}_{\infty,1}}ds
			&\lesssim   \|\tau_0\|_{L^2}\|\tau_0\|_{L^p} + \|\tau_0\|_{L^2}\|\Gamma\|_{L_t^\infty(L^p)}\\ \notag
			&~~~+\|\tau_0\|_{L^2}\left(\|u_0\|_{L^2} + \|\tau_0\|_{L^2}\right).
		\end{align}
		Combining with \eqref{2ineq3}, \eqref{2ineq5} and \eqref{2ineq10}, we conclude that
		\begin{align}\label{2ineq11}
			\int_0^t\|\tau\|_{B^0_{\infty,1}}ds &\lesssim \|\tau_0\|_{L^2} +  \|\tau_0\|_{L^2}\|\tau_0\|_{L^p} + \|\tau_0\|_{L^2}\|\Gamma\|_{L_t^\infty(L^p)}\\ \notag
			&~~~+\|\tau_0\|_{L^2}\left(\|u_0\|_{L^2} + \|\tau_0\|_{L^2}\right).
		\end{align}
		Assume that $\int_0^t\|\tau\|_{B^0_{\infty,1}}ds\leq c_1$ with $c_1\leq \frac {1} {4C}$, we infer from \eqref{2ineq2} and \eqref{con2} with $c\leq \frac {1} {4C}$ that
		\begin{align}\label{2ineq12}
			\|\Gamma\|_{L^p} &\leq C\|\Gamma_0\|_{L^p}
			+ C\|\tau_0\|_{L^p}.
		\end{align}
		Plugging \eqref{2ineq12} into \eqref{2ineq11}, one can arrive at
		\begin{align*}
			\int_0^t\|\tau\|_{B^0_{\infty,1}}ds &\leq C\|\tau_0\|_{L^2} +  C\|\tau_0\|_{L^2}\|\tau_0\|_{L^p} + C\|\tau_0\|_{L^2}\|\Gamma_0\|_{L^p}\\ \notag
			&~~~+C\|\tau_0\|_{L^2}\left(\|u_0\|_{L^2} + C\|\tau_0\|_{L^2}\right) \\ \notag
			&\leq 2Cc\leq \frac{c_1}{2},
		\end{align*}
		where we choose $c \in (0,\frac{c_1}{4C}]$. By continuity argument,  we conclude
		\begin{align}\label{2ineq13}
		\int_0^t\|\tau\|_{B^0_{\infty,1}}ds\leq c_1.
		\end{align}
		
		The last part is the boundness of $\|\tau\|_{\tilde{L}^1(0,\infty;B^{2-\frac{2}{p}}_{\infty,\infty})}$. By Lemma \ref{lemma1} and Bernstain's inequality, one can arrive at
		\begin{align}\label{2ineq14}
			\|\tau\|_{\tilde{L}^1(0,\infty;B^{2-\frac{2}{p}}_{\infty,\infty})} \lesssim \|\tau_0\|_{L^{p}} + \int_0^t \|Q(\Omega,\tau)\|_{B^{-\frac{2}{p}}_{\infty,\infty}}ds+\|u\cdot\nabla\tau\|_{\tilde{L}^1(0,\infty;B^{-\frac{2}{p}}_{\infty,\infty})}.
		\end{align}
		Applying Proposition \ref{prop1} with $p\in(2,\infty)$, one can arrive at
		\begin{align}\label{2ineq15}
			\int_0^t \|T_{\tau}\Omega\|_{B^{-\frac{2}{p}}_{\infty,\infty}}ds &\lesssim
			\int_0^t \|\Omega\|_{B^{-\frac{2}{p}}_{\infty,\infty}}\|\tau\|_{B^{0}_{\infty,1}}ds \\ \notag
			&\lesssim \|\Omega\|_{L^\infty_t(L^p)}\int_0^t \|\tau\|_{B^{0}_{\infty,1}}ds.
		\end{align}
		Applying Proposition \ref{prop2} leads to
		\begin{align}\label{2ineq16}
			\int_0^t \|R(\Omega,\tau)\|_{B^{-\frac{2}{p}}_{\infty,\infty}}ds &\lesssim
			\int_0^t \|R(\Omega,\tau)\|_{B^{0}_{p,\infty}}ds \\ \notag
			&\lesssim \|\Omega\|_{L^\infty_t(L^p)}\int_0^t \|\tau\|_{B^{0}_{\infty,1}}ds.
		\end{align}
		Thus, we deduce from \eqref{2ineq15} and \eqref{2ineq16} that
		\begin{align}\label{2ineq17}
			\int_0^t \|Q(\Omega,\tau)\|_{B^{-\frac{2}{p}}_{\infty,\infty}}ds \lesssim \|\nabla u\|_{L^\infty_t(L^p)}\int_0^t \|\tau\|_{B^{0}_{\infty,1}}ds.
		\end{align}
		Using Bony's decomposition, we get
		\begin{align}\label{2ineq18}
			\|u\cdot\nabla\tau\|_{\tilde{L}^1(0,\infty;B^{-\frac{2}{p}}_{\infty,\infty})} &\lesssim
			\|T_{u}\tau\|_{\tilde{L}^1(0,\infty;B^{1-\frac{2}{p}}_{\infty,\infty})}+\int_0^t \|T_{\tau}u\|_{B^{1-\frac{2}{p}}_{\infty,\infty}}+\|R(u,\tau)\|_{B^{1-\frac{2}{p}}_{\infty,\infty}}ds.
		\end{align}
		Applying Proposition \ref{prop1} leads to
		\begin{align}\label{2ineq19}
			\int_0^t \|T_{\tau}u\|_{B^{1-\frac{2}{p}}_{\infty,\infty}} ds &\lesssim \int_0^t \|u\|_{B^{1-\frac{2}{p}}_{\infty,\infty}}\|\tau\|_{B^{0}_{\infty,1}}ds \\ \notag
			&\lesssim \int_0^t \|u\|_{B^{1}_{p,\infty}}\|\tau\|_{B^{0}_{\infty,1}}ds \\ \notag
			&\lesssim \left(\|u\|_{L^\infty_t(L^2)} + \|\nabla u\|_{L^\infty_t(L^p)}\right)\int_0^t\|\tau\|_{B^{0}_{\infty,1}}ds,
		\end{align}
		where we use the fact that
		\begin{align*}
			\|u\|_{B^1_{p,\infty}} &\lesssim \|\Delta_1u\|_{L^p} + \sup_{j\geq0}2^{j}\|\Delta_ju\|_{L^p} \\ \notag
			&\lesssim \|u\|_{L^2} + \|\nabla u\|_{L^p}.
		\end{align*}
		Analogously, one can arrive at
		\begin{align}\label{2ineq20}
			\int_0^t \|R(u,\tau)\|_{B^{1-\frac{2}{p}}_{\infty,\infty}} ds &\lesssim \int_0^t \|R(u,\tau)\|_{B^{1}_{p,\infty}}ds \\ \notag
			&\lesssim \int_0^t \|u\|_{B^{1}_{p,\infty}}\|\tau\|_{B^{0}_{\infty,1}}ds \\ \notag
			&\lesssim \left(\|u\|_{L^\infty_t(L^2)} + \|\nabla u\|_{L^\infty_t(L^p)}\right)\int_0^t\|\tau\|_{B^{0}_{\infty,1}}ds.
		\end{align}
		For any $p\in(2,\infty)$, we have
		\begin{align*}
		\|u\|_{L^\infty} &\lesssim \|\Delta_1u\|_{L^\infty} + \sum_{j\geq0}\|\Delta_ju\|_{L^\infty} \\ \notag
		&\lesssim \|\Delta_1u\|_{L^2} + \sum_{j\geq0}2^{\frac 2 p j}\|\Delta_ju\|_{L^p} \\ \notag
		&\lesssim \|u\|_{L^2} + \|\nabla u\|_{L^p}.
		\end{align*}
		Applying Proposition \ref{prop1} leads to
		\begin{align}\label{2ineq21}
		\|T_{u}\tau\|_{\tilde{L}^1(0,\infty;B^{1-\frac{2}{p}}_{\infty,\infty})} &\lesssim  \|u\|_{L^\infty_t(L^\infty)}\|\tau\|_{\tilde{L}^1(0,\infty;B^{1-\frac{2}{p}}_{\infty,\infty})} \\ \notag
		&\lesssim \left(\|u\|_{L^\infty_t(L^2)} + \|\nabla u\|_{L^\infty_t(L^p)}\right)\|\tau\|^{\frac 1 2}_{\tilde{L}^1(0,\infty;B^{-\frac{2}{p}}_{\infty,\infty})}\|\tau\|^{\frac 1 2}_{\tilde{L}^1(0,\infty;B^{2-\frac{2}{p}}_{\infty,\infty})} \\ \notag
		&\lesssim \left(\|u\|_{L^\infty_t(L^2)} + \|\nabla u\|_{L^\infty_t(L^p)}\right)(\int_0^t\|\tau\|_{H^1}ds)^{\frac 1 2}\|\tau\|^{\frac 1 2}_{\tilde{L}^1(0,\infty;B^{2-\frac{2}{p}}_{\infty,\infty})}\\ \notag
		&\lesssim \left(\|u\|_{L^\infty_t(L^2)} + \|\nabla u\|_{L^\infty_t(L^p)}\right)\|\tau_0\|^{\frac 1 2}_{L^2}\|\tau\|^{\frac 1 2}_{\tilde{L}^1(0,\infty;B^{2-\frac{2}{p}}_{\infty,\infty})}\\ \notag
		&\lesssim \left(\|u\|^2_{L^\infty_t(L^2)} + \|\nabla u\|^2_{L^\infty_t(L^p)}\right)\|\tau_0\|_{L^2}+c\|\tau\|_{\tilde{L}^1(0,\infty;B^{2-\frac{2}{p}}_{\infty,\infty})}.
		\end{align}
		Thus, we deduce from \eqref{2ineq18}-\eqref{2ineq21} that
		\begin{align}\label{2ineq22}
			\int_0^t\|u\cdot\nabla\tau\|_{B^{-\frac{2}{p}}_{\infty,\infty}}ds &\lesssim
			\left(\|u\|_{L^\infty_t(L^2)} + \|\nabla u\|_{L^\infty_t(L^p)}\right)\int_0^t\|\tau\|_{B^{0}_{\infty,1}}ds \\ \notag
			&~~~+\left(\|u\|^2_{L^\infty_t(L^2)} + \|\nabla u\|^2_{L^\infty_t(L^p)}\right)\|\tau_0\|_{L^2}+c\|\tau\|_{\tilde{L}^1(0,\infty;B^{2-\frac{2}{p}}_{\infty,\infty})}.
		\end{align}
		Combining \eqref{2ineq14} , \eqref{2ineq17} and \eqref{2ineq22}, we conclude that
		\begin{align}\label{3ineq24}
			\|\tau\|_{\tilde{L}^1(0,\infty;B^{2-\frac{2}{p}}_{\infty,\infty})} &\lesssim \|\tau_0\|_{L^{p}} +
			\left(\|u\|_{L^\infty_t(L^2)} + \|\nabla u\|_{L^\infty_t(L^p)}\right)\int_0^t\|\tau\|_{B^{0}_{\infty,1}}ds \\ \notag
			&~~~+\left(\|u\|^2_{L^\infty_t(L^2)} + \|\nabla u\|^2_{L^\infty_t(L^p)}\right)\|\tau_0\|_{L^2}\\ \notag
			&\lesssim \|u_0\|_{L^2\cap \dot{W}^{1,p}} + \|\tau_0\|_{L^{2}\cap L^{p}}.
		\end{align}
		This completes the proof of Proposition \ref{case2}.
	\end{proof}
	\subsection{The noncorotation case}
	Taking $a=\nu=0$ and $\alpha=\mu=b=1$ in \eqref{eq0}, we obtain the noncorotation inviscid Oldroyd-B model:
	\begin{align}\label{eq4}
	\left\{\begin{array}{l}
	\partial_tu + u\cdot\nabla u+\nabla {\rm P} = {\rm div}~\tau,~~~~{\rm div}~u=0,\\[1ex]
	\partial_t\tau + u\cdot\nabla\tau+Q(\nabla u,\tau)= D(u)+\Delta\tau,\\[1ex]
	u|_{t=0}=u_0,~~\tau|_{t=0}=\tau_0. \\[1ex]
	\end{array}\right.
	\end{align}
	We establish the energy estimate of \eqref{eq4} under the $H^1$ framework.
	\begin{prop}\label{5prop1}
		Assume $(u,\tau)$ is a smooth solution of \eqref{eq4} with $(u_0,\tau_0)\in H^1$. There exists some sufficiently small constant $c>0$ such that if
		\begin{align}\label{5ineq0}
		\|(u_0,\tau_0)\|_{H^1}\leq c,
		\end{align}
		then for any $t>0$, we have
		\begin{align}\label{5ineq1}
		\frac{d}{dt}\left(\|(u,\tau)\|^2_{H^1}-\eta\langle\tau,\nabla u\rangle\right) + \frac \eta 8\|\nabla u\|^2_{L^2} + \frac 1 2\|\nabla\tau\|^2_{H^1} \leq 0,
		\end{align}
		and
		\begin{align}\label{5ineq2}
		\frac{d}{dt}\|\nabla(u,\tau)\|^2_{L^2} + \|\nabla^2\tau\|^2_{L^2} \leq C\|\nabla u\|^2_{L^2}\|\tau\|^2_{H^1}.
		\end{align}
	\end{prop}
	\begin{proof}
		We firstly assume that $\|(u,\tau)\|_{L_t^\infty(H^1)}\leq 4c$.
		Taking $L^2$ inner product for $(\ref{eq4})_1$ with $u$, we have
		\begin{align}\label{50ineq1}
		\frac 1 2\frac{d}{dt}\|u\|^2_{L^2} = \langle{\rm div}~\tau,u\rangle.
		\end{align}
		Taking $L^2$ inner product for $(\ref{eq4})_2$ with $\tau$, we get
		\begin{align}\label{50ineq2}
		\frac 1 2\frac{d}{dt}\|\tau\|^2_{L^2} + \|\nabla\tau\|^2_{L^2} - \langle D u,\tau\rangle &= - \langle Q(\nabla u,\tau),\tau\rangle \\ \notag
		&\leq C\|\nabla u\|_{L^2}\|\tau\|^2_{L^4} \\ \notag
		& \leq C\|\nabla u\|_{L^2}\|\tau\|_{L^2}\|\nabla\tau\|_{L^2}\\ \notag
		& \leq C\|\nabla u\|^2_{L^2}\|\tau\|^2_{L^2}+\frac{1}{2}\|\nabla\tau\|^2_{L^2}\\ \notag
		& \leq \frac{c}{2}\|\nabla u\|^2_{L^2} + \frac{1}{2}\|\nabla\tau\|^2_{L^2},
		\end{align}
		where $c\leq \frac{1}{100C}$.
		We infer from \eqref{50ineq1} , \eqref{50ineq2} and cancellation $\langle{\rm div}~\tau,u\rangle+\langle D u,\tau\rangle=0$ that
		\begin{align}\label{50ineq3}
		\frac{d}{dt}\|(u,\tau)\|^2_{L^2} + \|\nabla\tau\|^2_{L^2}
		\leq c\|\nabla u\|^2_{L^2}.
		\end{align}
		
		Taking $L^2$ inner product for $(\ref{eq4})_2$ with $\nabla u$, we have
		\begin{align}\label{50ineq4}
		-\langle\tau_t,\nabla u\rangle + \frac{1}{2}\|\nabla u\|^2_{L^2} &=   \langle u\cdot\nabla\tau,\nabla u\rangle  + \langle Q(\nabla u,\tau),\nabla u\rangle-\langle \Delta \tau,\nabla u\rangle\\ \notag
		&\leq C\|u\|_{L^4}\|\nabla \tau\|_{L^4}\|\nabla u\|_{L^2} + C\|\nabla u\|^2_{L^2}\|\tau\|_{L^{\infty}}+C\|\nabla u\|_{L^2}\|\Delta\tau\|_{L^{2}}\\ \notag
		&\leq C\|u\|^{\frac{1}{2}}_{L^2}\|\nabla u\|^{\frac{3}{2}}_{L^2}\|\nabla\tau\|^{\frac{1}{2}}_{L^2}\|\nabla^2 \tau\|^{\frac{1}{2}}_{L^2} + C\|\nabla u\|^2_{L^2}\|\tau\|_{L^{2}}\\ \notag
		&~~~+C\|\nabla u\|^2_{L^2}\|\nabla^2\tau\|_{L^{2}}+C\|\nabla u\|_{L^2}\|\Delta\tau\|_{L^{2}}\\ \notag
		&\leq \frac 1 {10}\|\nabla u\|^2_{L^2} +C\|\nabla^2\tau\|^2_{L^2}.
		\end{align}
		Applying $\nabla$ to $(\ref{eq4})_1$ and taking $L^2$ inner product for with $\tau$, we have
		\begin{align}\label{50ineq5}
		-\langle\nabla u_t,\tau\rangle  &=   -\langle u\cdot\nabla u + {\rm div}~\tau ,{\rm div}~\tau\rangle  + \langle \nabla {\rm P},{\rm div}~\tau\rangle \\ \notag
		&\leq C\|u\|_{L^4}\|\nabla \tau\|_{L^4}\|\nabla u\|_{L^2} + C\|\nabla \tau\|^2_{L^2} \\ \notag
		&\leq C\|u\|^{\frac{1}{2}}_{L^2}\|\nabla u\|^{\frac{3}{2}}_{L^2}\|\nabla\tau\|^{\frac{1}{2}}_{L^2}\|\nabla^2 \tau\|^{\frac{1}{2}}_{L^2} + C\|\nabla\tau\|^2_{L^2} \\ \notag
		&\leq \frac 1 {10}\|\nabla u\|^2_{L^2} +C\|\nabla^2\tau\|^2_{L^2} +C\|\nabla\tau\|^2_{L^2}.
		\end{align}
		We infer from \eqref{50ineq4} and \eqref{50ineq5} that
		\begin{align}\label{50ineq6}
		-\frac{d}{dt}\langle\tau,\nabla u\rangle + \frac{1}{4}\|\nabla u\|^2_{L^2}
		\leq C\|\nabla\tau\|^2_{H^1}.
		\end{align}
		Adding up $\eta\times$\eqref{50ineq6} and \eqref{50ineq3}, we obtain
		\begin{align}\label{50ineq7}
		\frac{d}{dt}\left(\|(u,\tau)\|^2_{L^2}-\eta\langle\tau,\nabla u\rangle\right) + \frac{1}{2}\|\nabla\tau\|^2_{L^2} + \frac{\eta}{4}\|\nabla u\|^2_{L^2}
		\leq \frac 1 2\|\nabla^2\tau\|^2_{L^2},
		\end{align}
		where $\eta \leq \frac 1 {2C}$.
		
		Notice that $ \langle  u\cdot\nabla u,\Delta u\rangle=0$ for $d=2$. Then, we can establish the energy estimate of \eqref{eq4} under the $H^1$ framework.
		Taking $L^2$ inner product for $(\ref{eq4})_1$ with $\Delta u$, we have
		\begin{align}\label{50ineq8}
		\frac 1 2\frac{d}{dt}\|\nabla u\|^2_{L^2} =-\langle{\rm div}~\tau,\Delta u\rangle.
		\end{align}
		Taking $L^2$ inner product for $(\ref{eq4})_2$ with $\Delta\tau$, we have
		\begin{align}\label{50ineq9}
		\frac 1 2\frac{d}{dt}\|\nabla\tau\|^2_{L^2} + \|\nabla^2\tau\|^2_{L^2} + \langle D u,\Delta\tau\rangle &=   \langle  u\cdot\nabla\tau,\Delta\tau\rangle + \langle Q(\nabla u,\tau),\Delta\tau\rangle \\ \notag
		&\leq C\|u\|_{L^4}\|\nabla\tau\|_{L^4}\|\Delta\tau\|_{L^2} + C\|\nabla u\|_{L^2}\|\tau\|_{L^{\infty}}\|\Delta\tau\|_{L^2} \\ \notag
		& \leq C\|u\|^{\frac{1}{2}}_{L^2}\|\nabla u\|^{\frac{1}{2}}_{L^2}\|\nabla\tau\|^{\frac{1}{2}}_{L^2}\|\nabla^2 \tau\|^{\frac{3}{2}}_{L^2}\\ \notag
		&~~~+C\|\nabla u\|_{L^2}\|\tau\|_{H^2}\|\nabla^2\tau\|_{L^2}\\ \notag
		& \leq C\|\nabla u\|^2_{L^2}\|\tau\|^2_{H^1} + \frac{1}{2}\|\nabla^2\tau\|^2_{L^2}.
		\end{align}
		We infer from \eqref{50ineq8} ,\eqref{50ineq9} and cancellation $\langle{\rm div}~\tau,\Delta u\rangle+\langle D u,\Delta\tau\rangle=0$ that
		\begin{align}\label{50ineq10}
		\frac{d}{dt}\|\nabla(u,\tau)\|^2_{L^2} + \|\nabla^2\tau\|^2_{L^2} \leq C\|\nabla u\|^2_{L^2}\|\tau\|^2_{H^1}.
		\end{align}
		Together with \eqref{50ineq7} and \eqref{50ineq10}, we conclude that
		\begin{align}\label{50ineq11}
		\frac{d}{dt}\left(\|(u,\tau)\|^2_{H^1}-\eta\langle\tau,\nabla u\rangle\right) + \frac \eta 8\|\nabla u\|^2_{L^2} + \frac 1 2\|\nabla\tau\|^2_{H^1} \leq 0,
		\end{align}
		where $c\leq \eta$. Taking $\eta$ small enough, we deduce from \eqref{5ineq0} and \eqref{50ineq11} that $\|(u,\tau)\|_{L_t^\infty(H^1)}\leq 2c$.
		We thus complete the proof of Proposition \ref{5prop1}.
	\end{proof}	
	\section{Global existence and energy conservation.}
	In this section we firstly consider the global weak solutions of \eqref{eq2} with the boundness of $\nabla u$ in $L^p$. We focus on the proof of Theorem \ref{theo1}, while Theorems \ref{theo} and \ref{theo2} can be deduced in the same way.
	
	Here, we omit the construction of approximate solution sequence $(u^n, \tau^n)$, one can refer to \cite{2002Majda,Masmoudi2013} for more details. Note that the constructed approximate solution sequence satisfies the energy estimates in Proposition \ref{case1} or Proposition \ref{case2} for different $p$.
	We need to obtain the following weak convergence:
	$$
	{\rm div}~(u^n\otimes u^n)\rightharpoonup{\rm div}~(u\otimes u) ~~~\text{in}~~~ \mathscr{D}'([0,T]\times\mathbb{R}^2),
	$$
	$${\rm div}~(u^n\otimes \tau^n)\rightharpoonup{\rm div}~(u\otimes \tau) ~~~\text{in}~~~ \mathscr{D}'([0,T]\times\mathbb{R}^2),
	$$
	$$
	Q(\Omega^n,\tau^n)\rightharpoonup Q(\Omega,\tau) ~~~\text{in}~~~ \mathscr{D}'([0,T]\times\mathbb{R}^2),
	$$
	as $n$ tends to infinity.
	Firstly, we introduce some compactness lemmas.
	\begin{lemm}\cite{1998Evans,2004Feireisl}\label{imbedding}
		Let $K\subset\mathbb{R}^d$ be a bounded domain. \\
		(1) If $kp<d$ and $p\geq1$, the space $W_0^{k,p}(K)$ is continuously imbedded in $L^q(K)$ for any $1\leq q \leq p^{\ast}= \frac{dp}{d-kp}$. Moreover, the imbedding is compact if $k>0$ and $q<p^{\ast}$.\\
		(2) If $kp=d$, the space $W_0^{k,p}(K)$ is compactly imbedded in $L^q(K)$ for any finite $q$.\\
		(3) If $kp>d+\nu$ with $\nu>0$, then $W^{k,p}_0(K)$ is compactly imbedded in $C^{0,\nu}(\bar{K})$.
	\end{lemm}
	\begin{lemm}\cite{1996Lions,1998Lions}[Lions-Aubin]\label{Lions-Aubin}
		Let $\{f^{n}(t)\}$ be a sequence in $C(0,T;H^s(\mathbb{R}^d))$ such that\\
		(1) $\mathop{\max}\limits_{t\in[0,T]}\|f^n(t)\|_{H^s}\leq C$, \\
		(2) for any $\rho \in C^{\infty}_{c}(\mathbb{R}^d)$, $\{\rho f^n\}$ is uniformly in ${\rm Lip}(0,T;H^{M}(\mathbb{R}^d))$, i.e.,
		$$
		\| \rho f^{n}(t_1) - \rho f^{n}(t_2)\|_{H^M} \leq C_M |t_1-t_2|,
		$$
		for some constant $C_M$ and any $t_1,~t_2\in[0,T]$. Then there exists a subsequence $\{f^{n_j}\}$ and $f\in C(0,T;H^s(\mathbb{R}^d))$ such that for all $\alpha\in(M,s)$ and $\rho\in C^{\infty}_c(\mathbb{R}^d)$
		$$
		\mathop{\max}\limits_{t\in[0,T]}\|\rho f^{n_j}(t) - \rho f(t)\|_{H^\alpha} \rightarrow 0,
		$$
		as $j$ tends to infinity.
	\end{lemm}
	Then, we prove the following strong convergence property for $u$ and $\Omega$.
	\begin{prop}\label{prop6}
		If sequence $\{(u^n,\tau^n)\}_{n\in{\rm N}}$ is uniformly bounded in $C_T(L^2)$ and $\{\nabla u^n\}_{n\in{\rm N}}$ is uniformly bounded in $C_T(L^p)$ with $1<p<\infty$. then for any $\rho\in C^{\infty}_c(\mathbb{R}^2)$, there exist $u\in C_T(L^2)$ and $\Omega = \nabla\times u\in L_T^{\infty}(L^p)$ such that
		\begin{align}\label{4ineq0}
			\|\rho\Omega^n - \rho\Omega\|_{L^{\infty}_T(H^{-1})} \rightarrow 0,
		\end{align}
		as $n$ tends to infinity. Moreover, for any bounded domain $K\subset\mathbb{R}^2$, we have
		\begin{align}\label{4ineq1}
			\|u^n-u\|_{L^{\infty}_T(L^2(K))}\rightarrow 0,
		\end{align}
		as $n$ tends to infinity.
	\end{prop}
	\begin{proof}
		For simplicity, we denote all derivations of $\rho$ as $\Phi$. First of all, we claim
		\begin{align}\label{4ineq2}
			\|\rho\Omega^n(t_2)-\rho\Omega^n(t_1)\|_{H^{-L-2}}\leq C_{\rho}|t_2-t_1|,
		\end{align}
		for any $\rho\in C_c^{\infty}(\mathbb{R}^2)$ and $L>1$.
		
		A direct calculation leads to
		\begin{align}\label{4ineq3}
			\rho\Omega^n(t_2)-\rho\Omega^n(t_1)&=\rho\nabla\times(u^n(t_2)-u^n(t_1))\\ \notag
			&=\nabla\times(\rho u^n(t_2)-\rho u^n(t_1))-\nabla\rho\cdot(u^n(t_2)-u^n(t_1))^T.
		\end{align}
		We infer from $(\ref{eq2})_1$, \eqref{4ineq3} and Minkowski's inequality that
		\begin{align}\label{4ineq4}
			\|\rho\Omega^n(t_2)-\rho\Omega^n(t_1)\|_{H^{-L-2}} &\lesssim \|\Phi(u^n(t_2)-u^n(t_1))\|_{H^{-L-1}}\\ \notag
			&\lesssim \int_{t_1}^{t_2}\|\Phi\frac{d}{dt}u^n\|_{H^{-L-1}}ds\\ \notag
			&\lesssim |t_2-t_1|\left(\|\Phi \mathbb{P}{\rm div}~(u^n\otimes ^n)\|_{L^{\infty}_T(H^{-L-1})}+\|\Phi \mathbb{P}{\rm div}~\tau^n\|_{L^{\infty}_T(H^{-L-1})}\right).
		\end{align}
		According to the definition of Projection operator $\mathbb{P}$, we obtain
		\begin{align}\label{4ineq5}
			\Phi \mathbb{P}{\rm div}~(u^n\otimes ^n) &= \Phi {\rm div}~(u^n\otimes u^n)-\Phi \nabla \Delta^{-1}{\rm div}~{\rm div}~(u^n\otimes ^n)\\ \notag
			&={\rm div}~(\Phi(u^n\otimes u^n)) - (u^n\otimes u^n)\nabla\Phi - \nabla(\Phi \Delta^{-1}{\rm div}~{\rm div}~(u^n\otimes u^n)) \\ \notag
			&~~~+ \nabla\Phi \Delta^{-1}{\rm div}~{\rm div}~(u^n\otimes u^n).
		\end{align}
		By virtue of the dual of Lemma \ref{imbedding} and \eqref{4ineq5}, we deduce for $2\leq p<\infty$ that
		\begin{align}\label{4ineq6}
			\|\Phi \mathbb{P}{\rm div}~(u^n\otimes u^n)\|_{H^{-L-1}} &\lesssim \|\Phi(u^n\otimes u^n)\|_{H^{-L}} + \|\Phi\Delta^{-1}{\rm div}~{\rm div}~(u^n\otimes u^n)\|_{H^{-L}}\\ \notag
			&\lesssim \|\Phi(u^n\otimes u^n)\|_{L^1} + \|\Phi\Delta^{-1}{\rm div}~{\rm div}~(u^n\otimes u^n)\|_{L^{q}}\\ \notag
			&\lesssim \|\Phi\|_{L^{\infty}}\|u^n\|^2_{L^2} + \|\Phi\|_{L^{\infty}}\|u^n\|_{L^{2}}\|u^n\|_{L^{\frac{2q}{2-q}}}\\ \notag
			&\lesssim \|\Phi\|_{L^{\infty}}\|u^n\|^2_{L^2} + \|\Phi\|_{L^{\infty}}\|u^n\|_{L^{2}}\|\nabla u^n\|_{L^{p}},
		\end{align}
		where $1<q<2$. If $1<p<2$, then \eqref{4ineq6} is valid by taking $q=p$. We also infer by the definition of Projection operator $\mathbb{P}$ that
		\begin{align}\label{4ineq7}
			\Phi \mathbb{P}{\rm div}~\tau^n &= \Phi {\rm div}~\tau^n - \Phi\nabla\Delta^{-1}{\rm div}~{\rm div}~\tau^n \\ \notag
			&= {\rm div}(\Phi\tau^n) - \tau^n\nabla\Phi - \nabla(\Phi\Delta^{-1}{\rm div}~{\rm div}~\tau^n) + \Delta^{-1}{\rm div}~{\rm div}~\tau^n\nabla\Phi,
		\end{align}
		which implies that
		\begin{align}\label{4ineq8}
			\|\Phi \mathbb{P}{\rm div}~\tau^n\|_{H^{-L-1}} &\lesssim \|\Phi\tau^n\|_{H^{-L}} + \|\Phi\Delta^{-1}{\rm div}~{\rm div}~\tau^n\|_{H^{-L}} \\ \notag
			&\lesssim \|\Phi\|_{L^2}\|\tau^n\|_{L^2}.
		\end{align}
		Combining the estimates \eqref{4ineq4}, \eqref{4ineq6} and \eqref{4ineq8}, we prove the claim \eqref{4ineq2}. Moreover, for any $\rho\in C^{\infty}_c(\mathbb{R}^2)$ and $1<p<\infty$, there exists $s\in(0,1)$ such that
		\begin{align}\label{4ineq9}
			\|\rho\Omega^n\|_{C_T(H^{-s})}
			\leq C_{\rho}\|\Omega^n\|_{C_T(L^p)}.
		\end{align}
		We deduce from \eqref{4ineq2}, \eqref{4ineq9} and Lions-Aubin Lemma \ref{Lions-Aubin} that there exists $\Omega\in C_T(H^{-1})\cap L_T^\infty(L^p)$ such that
		\begin{align}\label{4ineq10}
			\|\rho\Omega^n - \rho\Omega\|_{L_T^\infty(H^{-1})} \rightarrow 0,
		\end{align}
		as $n$ tends to infinity.
		
		The Biot-Savart law implies that
		\begin{align}\label{4ineq14}
			u = \int_{\mathbb{R}^2}\mathscr{K}(x-y)\Omega(y,t)dy,
		\end{align}
		with convolution kernal $\mathscr{K}$ satisfying $|\mathscr{K}(x)|\leq\frac{C}{|x|}$. Take $\rho\in C^{\infty}_c(\mathbb{R}^2)$ such that $\rho\equiv1$ in $B(0,1)$ and $\rho\equiv0$ in $B^c(0,2)$. Denote that $\rho_\delta(|x|)=\rho(\frac{|x|}{\delta})$, then we have
		\begin{align}\label{4ineq15}
			u^n-u^m &= \int_{\mathbb{R}^2}\rho_\delta(x-y)\mathscr{K}(x-y)(\Omega^n-\Omega^m)(y,t)dy \\ \notag
			&~~~+ \int_{\mathbb{R}^2}(1-\rho_R(x-y))\mathscr{K}(x-y)(\Omega^n-\Omega^m)(y,t)dy \\ \notag
			&~~~+ \int_{\mathbb{R}^2}(\rho_R-\rho_\delta)(x-y)\mathscr{K}(x-y)(\Omega^n-\Omega^m)(y,t)dy\\ \notag
			&\triangleq \sum_{i=1}^3G_i.
		\end{align}
		For $G_1$, $x \in K$ and $x-y\in B(0,\delta)$ implies $y\in K+B(0,\delta)\triangleq K_\delta$. We infer from $\|\rho_\delta \mathscr{K}\|_{L^1}=\delta$ and Young's inequality that
		\begin{align}\label{4ineq16}
			\|G_1\|_{L^1(K)} &\leq \|\rho_\delta \mathscr{K}\|_{L^1}\|\Omega^n-\Omega^m\|_{L^1(K_\delta)}\\ \notag
			&\leq C_K\sup_{n\in N}\|\Omega^n\|_{L^p}\cdot\delta.
		\end{align}
		For $G_2$, we infer from $\|(1-\rho_R)\mathscr{K}\|_{L^{\infty}}\leq CR^{-1}$ and Young's inequality that
		\begin{align}\label{4ineq17}
			\|G_2\|_{L^{1}(K)} &\leq C_K\|G_2\|_{L^{\infty}(K)}\\ \notag
			&\leq C_K\|(1-\rho_R)\mathscr{K}\|_{L^{\infty}}\|\Omega^n-\Omega^m\|_{L^1(K_\delta)}\\ \notag
			&\leq C_K\sup_{n\in N}\|\Omega^n\|_{L^p}\cdot R^{-1}.
		\end{align}
		For $G_3$, $x \in K$ and $x-y\in B(0,R)$ implies $y\in K+B(0,R)\triangleq K_R$. We infer from $\|(\rho_R-\rho_\delta)\mathscr{K}\|_{H^1}\leq C_{R,\delta}$ and Young's inequality that
		\begin{align}\label{4ineq18}
			\|G_3\|_{L^{1}(K)} &\leq \|(\rho_R-\rho_\delta)\mathscr{K}\|_{H^1}\|\Omega^n-\Omega^m\|_{H^{-1}(K_R)}\\ \notag
			&\leq C_{K,R,\delta}\|\Omega^n-\Omega^m\|_{H^{-1}(K_R)}.
		\end{align}
		Combining \eqref{4ineq15}-\eqref{4ineq18}, we thus obtain
		\begin{align}\label{4ineq19}
			\|u^n-u^m\|_{L_T^{\infty}(L^1(K))} \leq C_K\left(\delta+R^{-1}\right) + C_{K,R,\delta} \|\Omega^n-\Omega^m\|_{L_T^{\infty}(H^{-1}(K_R))}.
		\end{align}
		According to \eqref{4ineq10} and \eqref{4ineq19}, for any $\varepsilon>0$, there exists a positive $N$ such that for any $n,m>N$,
		\begin{align}\label{4ineq20}
			\|u^n-u^m\|_{L^{\infty}_T(L^1(K))} \leq \varepsilon.
		\end{align}
		This implies that there exists $u\in L_T^\infty(L^1(K))$ such that
		\begin{align}\label{4ineq21}
		\|u^n-u\|_{L^{\infty}_T(L^1(K))}\rightarrow 0,
		\end{align}
		as $n$ tends to infinity. Take $\rho\in C^{\infty}_c(\mathbb{R}^2)$ such that $\rho\equiv1$ in $K$ and $\rho\equiv0$ in $(K+B(0,1))^c$.
		
		If $1<p<2$, we deduce from \eqref{4ineq21} that
		\begin{align}\label{4ineq22}
			\|u^n-u\|_{L^{\infty}_T(L^{2}(K))} &\leq C\|\rho(u^n-u)\|^{\frac{2p-2}{3p-2}}_{L^{\infty}_T(L^{1}(K))}\|\rho(u^n-u)\|^{\frac{p}{3p-2}}_{L^{\infty}_T(L^{\frac{2p}{2-p}}(K))}\\ \notag
			&\leq C\|u^n-u\|^{\frac{2p-2}{3p-2}}_{L^{\infty}_T(L^{1}(K+B(0,1)))}\sup_{n\in N}\|u^n\|^{\frac{p}{3p-2}}_{L^{\infty}_T(\dot{W}^{1,p}\cap L^2(K+B(0,1)))}\\ \notag
			&\leq C\|u^n-u\|^{\frac{2p-2}{3p-2}}_{L^{\infty}_T(L^{1}(K+B(0,1)))}\rightarrow 0,
		\end{align}
		as $n$ tends to infinity.
		
		If $2<p<\infty$, then there exists $1<q<2$ such that
		\begin{align}\label{4ineq23}
		\|u^n-u\|_{L^{\infty}_T(L^{2}(K))} &\leq C\|\rho(u^n-u)\|^{\frac{2q-2}{3q-2}}_{L^{\infty}_T(L^{1}(K))}\|\rho(u^n-u)\|^{\frac{q}{3q-2}}_{L^{\infty}_T(L^{\frac{2q}{2-q}}(K))}\\ \notag&\leq C\|u^n-u\|^{\frac{2q-2}{3q-2}}_{L^{\infty}_T(L^{1}(K+B(0,1)))}\sup_{n\in N}\|u^n\|^{\frac{q}{3q-2}}_{L^{\infty}_T(\dot{W}^{1,p}\cap L^2(K+B(0,1)))}\\ \notag
		&\leq C\|u^n-u\|^{\frac{2q-2}{3q-2}}_{L^{\infty}_T(L^{1}(K+B(0,1)))}\rightarrow 0,
		\end{align}
		as $n$ tends to infinity. The strong convergence property implies that $u\in C_T(L^2)$.
		This completes the proof of Proposition \ref{prop6}.
	\end{proof}
	\textbf{Proof of Theorems \ref{theo1} :}\\
	Assume $\phi\in\mathscr{D}([0,T)\times\mathbb{R}^2)$, we have
	\begin{align*}
		\int_0^T \int_{\mathbb{R}^2} Q(\Omega^n,\tau^n)\phi -  Q(\Omega,\tau)\phi dxdt &= \int_0^T \int_{\mathbb{R}^2} Q(\Omega^n-\Omega,\tau^n)\phi dxdt\\ \notag
		&~~~+\int_0^T \int_{\mathbb{R}^2} Q(\Omega,\tau^n-\tau)\phi dxdt \\ \notag
		&\leq C\|\phi\Omega^n-\phi\Omega\|_{L_T^{\infty}(H^{-1})}\mathop{\max}\limits_{n\in{\rm N}}\|\tau^n\|_{L_T^1(H^{1})} \\ \notag
		&~~~+\int_0^T \int_{\mathbb{R}^2} Q(\Omega,\tau^n-\tau)\phi dxdt.
	\end{align*}
	Then, we infer from Propositions \ref{case1}, \ref{case2}, \ref{prop6} and Lemma \ref{3lemma1} that
	\begin{align*}
	\int_0^T \int_{\mathbb{R}^2} Q(\Omega^n,\tau^n)\phi -  Q(\Omega,\tau)\phi dxdt\rightarrow 0,
	\end{align*}
	as $n$ goes to infinity. The following weak convergence of nonlinear terms
	$$
	{\rm div}~(u^n\otimes u^n)\rightharpoonup{\rm div}~(u\otimes u) ~~~\text{in}~~~ \mathscr{D}'([0,T]\times\mathbb{R}^2),
	$$
	and
	$${\rm div}~(u^n\otimes \tau^n)\rightharpoonup{\rm div}~(u\otimes \tau) ~~~\text{in}~~~ \mathscr{D}'([0,T]\times\mathbb{R}^2),
	$$
	are direct consequences for \eqref{4ineq1} and Lemma \ref{3lemma1}.
	We thus complete the proof of Theorem \ref{theo1}.
	\hfill$\Box$
	
	Moreover, we prove the energy conservation for weak
	solutions of the 2-D co-rotation inviscid Oldroyd-B model \eqref{eq2} under high integrability conditions.
	\begin{prop}\label{onsager}
		Let $(u,\tau)$ be a weak solution of \eqref{eq2} and satisfy Definition \ref{defi} with $a=\mu=1$. If $p>\frac 3 2$, then for any $t\in[0,T)$, we have
		\begin{align}\label{EC1}
		\|u(t)\|_{L^2}=\|u_0\|_{L^2}+2\int_{0}^{t}\int_{\mathbb{R}^2} u{\rm div}~\tau  dxds,
		\end{align}
		and
		\begin{align}\label{EC2}
		e^{2t}\|\tau(t)\|^2_{L^2} + 2\int_0^t e^{2s}\|\nabla\tau\|^2_{L^2} ds= \|\tau_0\|^2_{L^2}.
		\end{align}
	\end{prop}
	\begin{rema}\label{onsager1}
		If we change Definition \ref{defi} of weak solutions by integrating by parts of
		$$\int_{0}^{t}\int_{\mathbb{R}^2} Q(\Omega,\tau)\phi dxdt,$$
		then the condition $u\in L^{\infty}_T(\dot{W}^{1,p})$ can be removed.
		For $\alpha>
		\frac 1 3$, we can prove the energy conservation for weak
		solutions of \eqref{eq2} with extra regularity condition $u\in L^{3}_T(\dot{B}^{\alpha}_{3,\infty})$, which cover the Onsager's conjecture on the energy conservation \cite{On1} for the Euler equation. We omit the details here.
	\end{rema}
	\begin{proof}
		For the sake of simplicity, we will proceed as if the solution is differentiable in time. The extra arguments needed to mollify in time are straightforward.
		
		Let $J\in\mathscr{D}(\mathbb{R}^2)$ be a standard mollifier supported in $B(0,1)$ and $J_h(x)=\frac {1}{h^2}J(\frac {x}{h})$. For $f\in \mathscr{D}'(\mathbb{R}^2)$, we take $f^h=f\ast J_h$.
		If $f\in L^q$ with $q\in[1,\infty)$, then we have
		\begin{align*}
		\|f^h\|_{L^q}\leq \|f\|_{L^q},~~~~\lim\limits_{h\rightarrow 0}\|f^h-f\|_{L^q}=0.
		\end{align*}
		
		Under the assumption of differentiability in time, we obtain from Definition \ref{defi}
		\begin{align}\label{on0}
		\frac 1 2\frac d {dt}\|u^h\|^2_{L^2} =\int_{\mathbb{R}^2}(u\otimes u)^h:\nabla u^h + ({\rm div}~\tau)^h u^h dx.
		\end{align}
		We infer from \eqref{on0} that
		\begin{align}\label{on1}
		\|u^h(t)\|^2_{L^2}-\|u^h(0)\|^2_{L^2}-2\int_{0}^{t}\int_{\mathbb{R}^2} u^h({\rm div}~\tau)^h  dxds =2\int_{0}^{t}\int_{\mathbb{R}^2}(u\otimes u)^h:\nabla u^hdxds.
		\end{align}
		Denote $$r_h(f,g)=\int_{\mathbb{R}^2}(\delta_yf \delta_yg)J_h(y)dy,$$ where $\delta_y f(x)=f(x-y)-f(x)$. Then we deduce that $f-f^h=\int_{\mathbb{R}^2}\delta_yfJ_h(y)dy$ and
		\begin{align}\label{eq}
			(fg)^h=f^h g^h+r_h(f,g)-(f-f^h)(g-g^h).
	    \end{align}	
		This together with \eqref{on1} implies that
		\begin{align}\label{on2}
		\|u^h(t)\|^2_{L^2}-\|u^h(0)\|^2_{L^2}-2\int_{0}^{t}\int_{\mathbb{R}^2} u^h({\rm div}~\tau)^h  ds &\lesssim  \int_{0}^{t}\|r_h(u,u)\|_{L^{\frac 3 2}}\|\nabla u^h\|_{L^{3}}ds  \\ \notag
		&~~~+\int_{0}^{t}\int_{\mathbb{R}^2}\|u-u^h\|^2_{L^3}\|\nabla u^h\|_{L^{3}}dxds.
		\end{align}
		According to a characterization of Besov spaces $\dot{B}^{\alpha}_{3,\infty}$ with $\alpha\in(0,1)$ (see \cite{Bahouri2011}), then we have
		\begin{align}\label{on3}
		\|u\|_{\dot{B}^{\alpha}_{3,\infty}}\approx \|\frac {\|u(\cdot+y)-u(\cdot)\|_{L^p}}{|y|^{\alpha}}\|_{L^\infty}.
		\end{align}
		Combining \eqref{on2} and \eqref{on3}, we deduce that
		\begin{align}\label{on4}
		 \|u^h(t)\|^2_{L^2}-\|u^h(0)\|^2_{L^2}-2\int_{0}^{t}\int_{\mathbb{R}^2} u^h({\rm div}~\tau)^h  dxds &\lesssim h^{3\alpha-1}\int_{0}^{t}\|u\|^3_{\dot{B}^{\alpha}_{3,\infty}}ds.
		\end{align}
		
		If $p\in(\frac 3 2, 3]$, we infer that
		\begin{align}\label{on5}
		\|u\|_{\dot{B}^{\alpha}_{3,\infty}}&=\sup_{j\in \mathbb{Z}}2^{j\alpha}\|\dot{\Delta}_j u\|_{L^3} \\ \notag
		&\lesssim \sup_{j\leq 0}2^{j(\alpha+\frac 1 3)}\|\dot{\Delta}_j u\|_{L^2}+\sup_{j>0}2^{j(\alpha-1+\frac 2 p-\frac 2 3)}\|\dot{\Delta}_j \nabla u\|_{L^p} \\ \notag
		&\lesssim \|u\|_{L^2\cap \dot{W}^{1,p}},
		\end{align}
		where $\alpha=\frac 5 3-\frac 2 p\in(\frac 1 3,1)$ for $p\in(\frac 3 2, 3)$ and $\alpha=\frac 2 3$ for $p=3$.

		If $p\in(3, \infty)$, we deduce from the so-called Gagliardo-Nirenberg inequality (see \cite{Bahouri2011}) that
		\begin{align}\label{on6}
		\|u\|_{\dot{B}^{\alpha}_{3,\infty}}&\lesssim \|u\|_{\dot{W}^{\alpha,3}} \\ \notag
		&\lesssim \|u\|^{1-\alpha}_{L^2}\|\nabla u\|^{\alpha}_{\dot{W}^{1,p}} \\ \notag
		&\lesssim \|u\|_{L^2\cap \dot{W}^{1,p}},
		\end{align}
		where $\alpha=\frac p {3p-6}\in(\frac 1 3,1)$.
		Combining \eqref{on5} and \eqref{on6}, we obtain \eqref{EC1} by taking
		$h\rightarrow 0$ in \eqref{on4}.

    Under the assumption of differentiability in time, we obtain from Definition \ref{defi}
	\begin{align}\label{on7}
		\frac{1}{2}\frac{d}{dt}\|\tau^{h}\|^2_{L^2} + \|\tau^{h}\|^2_{L^2} + \|\nabla\tau^{h}\|^2_{L^2} = \langle(u\tau)^{h},\nabla\tau^{h}\rangle - \langle Q^{h}(\Omega,\tau),\tau^{h}\rangle.
	\end{align}
	For the first term on the right hand side of \eqref{on7}, we infer from \eqref{eq} that
	\begin{align}\label{on8}
		\langle(u\tau)^{h},\nabla\tau^{h}\rangle &= \langle\gamma_h(u,\tau) ,\nabla\tau^{h}\rangle - \langle(u^{h}-u)(\tau^{h}-\tau) ,\nabla\tau^{h}\rangle + \frac 1 2\langle u^{h} ,\nabla(\tau^{h})^2\rangle \\ \notag
		&= \langle\gamma_h(u,\tau) ,\nabla\tau^{h}\rangle - \langle(u^{h}-u)(\tau^{h}-\tau) ,\nabla\tau^{h}\rangle.
	\end{align}
	According to H{\"o}lder's inequality and \eqref{on3}, we have
	\begin{align*}
		\langle\gamma_h(u,\tau) ,\nabla\tau^{h}\rangle &\lesssim \|\gamma_h(u,\tau)\|_{L^2}\|\nabla\tau^{h}\|_{L^2} \\ \notag
		&\lesssim h^{\alpha}\|u\|_{\dot{B}^{\alpha}_{3,\infty}}\|\tau\|_{L^{6}}\|\nabla\tau\|_{L^2}\\ \notag
		&\lesssim  h^{\alpha}\|u\|_{L^2\cap\dot{W}^{1,p}}\|\tau\|^2_{H^1},
	\end{align*}
	and
	\begin{align*}
		\langle(u^{h}-u)(\tau^{h}-\tau) ,\nabla\tau^{h}\rangle &\lesssim \|u^{h}-u\|_{L^{3}}\|\tau^{h}-\tau\|_{L^{6}}\|\nabla\tau^{h}\|_{L^2} \\ \notag
		&\lesssim h^{\alpha}\|u\|_{\dot{B}^{\alpha}_{3,\infty}}\|\tau\|_{L^{6}}\|\nabla\tau\|_{L^2}\\ \notag
		&\lesssim h^{\alpha}\|u\|_{L^2\cap\dot{W}^{1,p}}\|\tau\|^2_{H^1}.
	\end{align*}
	Therefore, we deduce from \eqref{on8} that
	\begin{align}\label{on9}
		\langle(u\tau)^{h},\nabla\tau^{h}\rangle \lesssim h^{\alpha}\|u\|_{L^2\cap\dot{W}^{1,p}}\|\tau\|^2_{H^1}.
	\end{align}
	For the second term on the right hand side of \eqref{on7}, we infer from \eqref{eq} that
	\begin{align}\label{on10}
		\langle Q^{h}(\Omega,\tau),\tau^{h}\rangle &= \langle\gamma_h(\Omega,\tau),\tau^{h}\rangle - \langle (\Omega^{h}-\Omega)(\tau^{h}-\tau) ,\tau^{h}\rangle + \langle \Omega^{h} \tau^{h},\tau^{h}\rangle \\ \notag
		&= \langle\gamma_h(\Omega,\tau),\tau^{h}\rangle - \langle (\Omega^{h}-\Omega)(\tau^{h}-\tau) ,\tau^{h}\rangle.
	\end{align}
	Integrating by parts of \eqref{on10} and using H{\"o}lder's inequality, we have
	\begin{align}\label{on11}
		\langle Q^{h}(\Omega,\tau),\tau^{h}\rangle &\lesssim h^{\alpha}\|u\|_{\dot{B}^{\alpha}_{3,\infty}}\|\tau\|_{\dot{H}^{1}}\|\tau\|_{L^{6}} \\ \notag
		&\lesssim h^{\alpha}\|u\|_{L^2\cap\dot{W}^{1,p}}\|\tau\|^2_{H^1}.
	\end{align}
	According to \eqref{on7}, \eqref{on9} and \eqref{on11}, we conclude that
	\begin{align}\label{on12}
		\frac{1}{2}\frac{d}{dt}\|\tau^{h}\|^2_{L^2} + \|\tau^{h}\|^2_{L^2} + \|\nabla\tau^{h}\|^2_{L^2}\lesssim h^{\alpha}.
	\end{align}
    Finally, we obtain \eqref{EC2} by taking
    $h\rightarrow 0$ in \eqref{on12}. This completes the proof of Proposition \ref{onsager}.
	\end{proof}
	
	\section{Optimal dacay rate}
	We investigate optimal decay rate of global weak solutions for the 2-D noncorotation inviscid Oldroyd-B equation \eqref{eq4} in this section. Due to low regularity of the solutions, this critical case $d=2$ without damping $(a=0)$ is extremely challenging. Applying the Fourier splitting method, the authors \cite{DLY1} studied long time behaviour of global weak solutions for \eqref{eq0} with $\nu=0$ and $a>0$. Moreover, time decay rate of the $\dot{H}^{1}$ norm was not optimal because of the large oscillation in the higher-order derivative of the solutions.
	
	We will solve the difficulties of damping disappearance and low regularity by virtue of the improved Fourier splitting method.
	Firstly we introduce the energy and energy dissipation functionals for $(u,\tau)$ as follows:
	\begin{align*}
		&E_{\eta} = \|(u,\tau)\|^2_{H^1}-\eta\langle\tau,\nabla u\rangle,\\ \notag
		&H_{\eta} = \frac{\eta}{16}\|\nabla u\|^2_{L^2} + \frac{1}{4}\|\nabla\tau\|^2_{H^1}.
	\end{align*}
	Moreover, the difficult terms in the estimate of optimal decay rate for \eqref{eq4} are denoted by
	\begin{align*}
		&B_1 = \int_0^t \|u\|^3_{L^2} + \|u\|^2_{L^2}\|\tau\|^2_{L^2} ds,\\ \notag
		&B_2 = \int_0^t(\|\nabla u\|_{L^2}\|\tau\|_{L^2}\int_{S(t)}|\hat{\tau}|d\xi) ds.
	\end{align*}

	Based on Schonbek's strategy \cite{Schonbek1991}, we introduce the Fourier splitting method in the following lemma. By a density argument, we only need to prove that the conclusions in this section hold for the smooth solutions.
	\begin{lemm}\label{5lemma1}
		Let $S(t)=\left\{\xi:|\xi|^2\leq C_2\frac{f'(t)}{f(t)}\right\}$ with $f(t)=\ln^{l}(e+t)$  or $f(t)=(1+t)^l$ for some $l\in {\rm N}$. Assume that $(u,\tau)$ satisfies the conditions in Theorem \ref{theo2}, then we have
		\begin{align}\label{51ineq0}
			\frac{d}{dt}E_{\eta} + H_{\eta} + \frac{C_2(\eta+1)f'(t)}{16f(t)}\|(u,\tau)\|^2_{L^2} \leq C\left(\frac{f'(t)}{f(t)}\right)^2\left(\|(u_0,\tau_0)\|_{\dot{B}^{-1}_{2,\infty}}+B_1\right) + C\frac{f'(t)}{f(t)}B_2.
		\end{align}
	\end{lemm}
	\begin{proof}
		We firstly rewrite \eqref{5ineq1} as
		\begin{align*}
		\frac{d}{dt}E_{\eta} + 2H_{\eta}  \leq 0.
		\end{align*}
		This together with the Schonbek's strategy \cite{Schonbek1991} implies that
		\begin{align}\label{51ineq1}
			\frac{d}{dt}E_{\eta} + H_{\eta} + \frac{C_2(\eta+1)f'(t)}{16f(t)}\int_{\mathbb{R}^2} |\hat{u}|^2 + |\hat{\tau}|^2 d\xi \lesssim\frac{f'(t)}{f(t)}\int_{S(t)} |\hat{u}|^2 + |\hat{\tau}|^2 d\xi.
		\end{align}
		Taking Fourier transform $\mathscr{F}$ to \eqref{eq4}, we have
		\begin{align}\label{4eq1}
			\left\{\begin{array}{l}
				\frac{d}{dt}\hat{u} + i\xi^{T}\mathscr{F}(u \otimes u) + i\xi\hat{P} = i\xi^{T}\hat{\tau},\\
				\frac{d}{dt}\hat{\tau} +\mathscr{F}(u \cdot\nabla\tau) + |\xi|^2\hat{\tau} +\mathscr{F}Q(\nabla u,\tau)= \frac{i}{2}( \xi \otimes \hat{u} + \hat{u} \otimes \xi ).
			\end{array}\right.
		\end{align}
		Multiplying \eqref{4eq1} by $(\bar{\hat{u}},\bar{\hat{\tau}})$ and taking the real part, we deduce that
		\begin{align*}
			\frac{1}{2}\frac{d}{dt}|\hat{u}|^2  = \mathcal{R}e[-i\xi^{T}\mathscr{F}(u \otimes u)\bar{\hat{u}}+i\xi^{T}\hat{\tau}\bar{\hat{u}}],
		\end{align*}
		and
		\begin{align*}
			\frac{1}{2}\frac{d}{dt}|\hat{\tau}|^2 + |\xi|^2|\hat{\tau}|^2=\mathcal{R}e[-\mathscr{F}(u \cdot\nabla\tau):\bar{\hat{\tau}}-\mathscr{F}Q(\nabla u,\tau):\bar{\hat{\tau}}+\frac{i}{2}(\xi \otimes \hat{u} + \hat{u} \otimes \xi ):\bar{\hat{\tau}}].
		\end{align*}
		Since $\tau$ is symmetric, we have
		\begin{align*}
			\mathcal{R}e[i\xi^{T}\hat{\tau}\bar{\hat{u}} + \frac{i}{2}( \xi \otimes \hat{u} + \hat{u} \otimes \xi ):\bar{\hat{\tau}}] = 0,
		\end{align*}
		which implies that
		\begin{align}\label{51ineq2}
			\frac{1}{2}\frac{d}{dt}(|\hat{u}|^2 + |\hat{\tau}|^2 ) + |\xi|^2|\hat{\tau}|^2 &= \mathcal{R}e[- i\xi^{T}\mathscr{F}(u \otimes u)\bar{\hat{u}} -\mathscr{F}(u \cdot\nabla\tau):\bar{\hat{\tau}} -\mathscr{F}Q(\nabla u,\tau):\bar{\hat{\tau}}] \\ \notag
			&\leq |\xi||\mathscr{F}(u \otimes u)||\hat{u}|+|\xi||\mathscr{F}(u \otimes\tau)||\hat{\tau}| + |\mathscr{F}Q(\nabla u,\tau)||\hat{\tau}|\\ \notag
			&\leq |\xi||\mathscr{F}(u \otimes u)||\hat{u}|+|\mathscr{F}(u \otimes\tau)|^2 + |\mathscr{F}Q(\nabla u,\tau)||\hat{\tau}|+\frac{1}{2}|\xi|^2|\hat{\tau}|^2.
		\end{align}
		Integrating \eqref{51ineq2} over $[0,t]$, we obtain
		\begin{align}\label{51ineq3}
			|\hat{u}|^2 + |\hat{\tau}|^2 \lesssim |\widehat{u_0}|^2 + |\widehat{\tau_0}|^2 + \int_0^t|\xi||\mathscr{F}(u \otimes u)||\hat{u}|+|\mathscr{F}(u \otimes\tau)|^2 + |\mathscr{F}Q(\nabla u,\tau)||\hat{\tau}|ds.
		\end{align}
		Integrating \eqref{51ineq3} over $S(t)$, we infer from Proposition \ref{prop0} that
		\begin{align}\label{51ineq4}
			\int_{S(t)} |\hat{u}|^2 + |\hat{\tau}|^2 d\xi &\lesssim \int_{S(t)}|\widehat{u_0}|^2 + |\widehat{\tau_0}|^2 d\xi + \int_{S(t)}\int_0^t|\xi||\mathscr{F}(u \otimes u)||\hat{u}|+|\mathscr{F}(u \otimes\tau)|^2dsd\xi \\ \notag
			&~~~+ \int_{S(t)}\int_0^t|\mathscr{F}Q(\nabla u,\tau)||\hat{\tau}|dsd\xi  \\ \notag
			&\lesssim \sum_{j\leq \log_2[\frac {4} {3}C_2^{\frac 1 2 }\sqrt{\frac {f'(t)}{f(t)}}]}\int_{\mathbb{R}^{d}} \varphi^2(2^{-j}\xi)(|\hat{u}_0|^2+|\hat{\tau}_0|^2)d\xi \\ \notag
			&~~~+ \int_0^t(\|u\|^2_{L^2}\int_{S(t)}|\xi||\hat{u}|d\xi+\|u\otimes\tau\|^2_{L^1}\int_{S(t)}d\xi) ds \\ \notag
			&~~~+ \int_0^t(\|Q(\nabla u,\tau)\|_{L^1}\int_{S(t)}|\hat{\tau}|d\xi) ds \\ \notag
			&\lesssim \frac{f'(t)}{f(t)}\|(u_0,\tau_0)\|^2_{\dot{B}^{-1}_{2,\infty}} + \frac{f'(t)}{f(t)}\int_0^t\|u\|^3_{L^2}+\|u\|^2_{L^2}\|\tau\|^2_{L^2}ds \\ \notag
			&~~~+ \int_0^t(\|\nabla u\|_{L^2}\|\tau\|_{L^2}\int_{S(t)}|\hat{\tau}|d\xi) ds \\ \notag
			&\lesssim \frac{f'(t)}{f(t)}\left(\|(u_0,\tau_0)\|^2_{\dot{B}^{-1}_{2,\infty}} + B_1\right) + B_2.
		\end{align}
		Combining \eqref{51ineq1} and \eqref{51ineq4}, we thus complete the proof of Lemma \ref{5lemma1}.
	\end{proof}
	For convenience, we denote that
	$$E_0 = \|(u,\tau)\|^2_{H^1},~~~~E_1 = \|\nabla(u,\tau)\|^2_{L^2}.$$
	In critical case $d=2$, we only obtain the logarithmic decay rate by the Fourier splitting method. However, we fail to improve time decay rate of the $\dot{H}^{1}$ norm of solutions with low regularity. Then, we derive time weighted integrability of the $\dot{H}^{1}$ norm of solutions.
	\begin{prop}\label{5prop2}
		Assume that $(u,\tau)$ satisfies the conditions in Theorem \ref{theo2}, then for any integer $l\in {\rm N}$, there exists a positive constant $C=C(l)$ such that
		\begin{align*}
			E_0 \leq C\ln^{-l}(e+t),
		\end{align*}
		and
		\begin{align*}
			\int_0^t \ln^{l}(e+s)E_1ds \leq C.
		\end{align*}
	\end{prop}
	\begin{proof}
		Denote $\tilde{B}_2 = \int_0^t\|\nabla u\|_{L^2}\|\tau\|^2_{L^2}ds,$
		then we obtain
		\begin{align}\label{52ineq0}
		B_2 \lesssim \left(\frac{f'(t)}{f(t)}\right)^{\frac{1}{2}}\tilde{B}_2.
		\end{align}
		Multiplying \eqref{51ineq0} by $f(t)$, we infer from \eqref{52ineq0} that
		\begin{align}\label{52ineq1}
			\frac{d}{dt}\left(f(t)E_{\eta}\right) + f(t)H_{\eta} + \frac{C_2(\eta+1)f'(t)}{16}\|(u,\tau)\|^2_{L^2}  &\leq C\frac{f'^2(t)}{f(t)}\left(\|(u_0,\tau_0)\|_{\dot{B}^{-1}_{2,\infty}}+B_1\right) \\ \notag
			&~~~+ C\frac{f'^{\frac{3}{2}}(t)}{f^{\frac{1}{2}}(t)}\tilde{B}_2
			+ f'(t)E_{\eta}.
		\end{align}
		Taking $C_2$ and $t$ sufficiently large, we infer that
		\begin{align}\label{52ineq2}
			\frac{d}{dt}\left(f(t)E_{\eta}\right) &\lesssim \frac{f'^2(t)}{f(t)}\left(\|(u_0,\tau_0)\|_{\dot{B}^{-1}_{2,\infty}}+B_1\right) + \frac{f'^{\frac{3}{2}}(t)}{f^{\frac{1}{2}}(t)}\tilde{B}_2.
		\end{align}
		Choosing $f(t)=\ln^3(e+t)$ in \eqref{52ineq2} and using Proposition \ref{5prop1}, we deduce that
		\begin{align}\label{52ineq3}
			\frac{d}{dt}\left(\ln^3(e+t)E_{\eta}\right) &\lesssim \frac{\ln(e+t)}{(1+t)^2}\left(\|(u_0,\tau_0)\|_{\dot{B}^{-1}_{2,\infty}}+B_1\right) + \frac{\ln^{\frac{3}{2}}(e+t)}{(1+t)^{\frac{3}{2}}}\tilde{B}_2 \\ \notag
			&\lesssim \frac{\ln(e+t)}{(1+t)^2}\|(u_0,\tau_0)\|_{\dot{B}^{-1}_{2,\infty}} + \frac{\ln^{\frac{3}{2}}(e+t)}{1+t}.
		\end{align}
		Integrating \eqref{52ineq3} over $[0,t]$, we infer that
		\begin{align*}
			\ln^3(e+t)E_{\eta} \lesssim \ln^{\frac{5}{2}}(e+t),
		\end{align*}
		which implies that $E_{0} \lesssim \ln^{-\frac{1}{2}}(e+t)$ by the
		equivalence $E_{\eta} \sim E_0$.
		
		Then, we will demonstrate the logarithmic decay of any order by the induction argument. Assume
		$$
		E_0 \lesssim \ln^{-\frac{l}{2}}(e+t),
		$$
		for any $l\in{\rm N}$. Choosing $f(t)=\ln^{l+1}(e+t)$ in \eqref{52ineq2} and using the assumption, we obtain
		\begin{align}\label{52ineq4}
			\frac{d}{dt}\left(\ln^{l+1}(e+t)E_{\eta}\right) &\lesssim \frac{\ln^{l-1}(e+t)}{(1+t)^2}\left(\|(u_0,\tau_0)\|_{\dot{B}^{-1}_{2,\infty}}+B_1\right) + \frac{\ln^{l-\frac{1}{2}}(e+t)}{(1+t)^{\frac{3}{2}}}\tilde{B}_2 \\ \notag
			&\lesssim \frac{\ln^{l-1}(e+t)}{(1+t)^2}\|(u_0,\tau_0)\|_{\dot{B}^{-1}_{2,\infty}} + \frac{\ln^{\frac{l}{2}-\frac{1}{2}}(e+t)}{1+t},
		\end{align}
		where we use the following inequality
		\begin{align*}
			(1+t)^{-1}\int_0^t \ln^{-l}(e+s)ds\lesssim \ln^{-l}(e+t),~~~\forall~l\in\mathbb{N}.
		\end{align*}
		Integrating \eqref{52ineq4} over $[0,t]$, we infer that
		\begin{align*}
			\ln^{l+1}(e+t)E_{\eta} \lesssim \ln^{\frac{l}{2}+\frac{1}{2}}(e+t),
		\end{align*}
		which implies that $E_{0} \lesssim \ln^{-\frac{l}{2}-\frac{1}{2}}(e+t).$
		Thus, we conclude by induction that
		\begin{align}\label{52ineq5}
		E_{0} \lesssim \ln^{-l}(e+t)
		\end{align}
		for any $l\in{\rm N}$. Moreover, choosing $f(t)=\ln^{l}(e+t)$ in \eqref{52ineq1} and using \eqref{52ineq5}, we deduce that
		\begin{align}\label{52ineq6}
			\frac{d}{dt}\left(\ln^{l}(e+t)E_{\eta}\right) + \frac 1 2\ln^{l}(e+t)H_{\eta}  &\lesssim \frac{\ln^{l-2}(e+t)}{(1+t)^2}\left(\|(u_0,\tau_0)\|_{\dot{B}^{-1}_{2,\infty}}+B_1\right) + \frac{\ln^{l-\frac{3}{2}}(e+t)}{(1+t)^{\frac{3}{2}}}\tilde{B}_2 \\ \notag
			&\lesssim \frac{\ln^{l-2}(e+t)}{(1+t)^2}\|(u_0,\tau_0)\|_{\dot{B}^{-1}_{2,\infty}} + \frac{\ln^{-\frac{3}{2}}(e+t)}{1+t}.
		\end{align}
		Integrating \eqref{52ineq6} over $[0,t]$, we infer that
		\begin{align*}
			\ln^{l}(e+t)E_{\eta} + \int_0^t\ln^{l}(e+s)H_{\eta}ds \lesssim 1,
		\end{align*}
		which implies $\int_0^t\ln^{l}(e+s)E_1ds \lesssim 1$ by the definition of $E_1$ and $H_{\eta}$. This completes the proof of Proposition \ref{5prop2}.
	\end{proof}
	As the first step in the iterative process to prove optimal decay rate, we derive the initial algebraic decay rate by the time weighted energy estimate and the logarithmic decay rate in Proposition \ref{5prop2}. Moreover, we improve time decay rate of the $\dot{H}^{1}$ norm by time weighted integrability of the $\dot{H}^{1}$ norm.
	\begin{prop}\label{5prop3}
		Assume that $(u,\tau)$ satisfies the conditions in Theorem \ref{theo2}, then we have
		\begin{align}
			E^2_0 + E_1 \leq C(1+t)^{-1}.
		\end{align}
	\end{prop}
	\begin{proof}
		Choosing $f(t)=(1+t)^2$ in \eqref{52ineq1}, we obtain
		\begin{align}\label{53ineq1}
			\frac{d}{dt}\left((1+t)^2E_{\eta}\right) \lesssim \|(u_0,\tau_0)\|_{\dot{B}^{-1}_{2,\infty}}+B_1 + (1+t)^{\frac{1}{2}}\tilde{B}_2,
		\end{align}
		where $B_1=\int_0^t\|u\|^3_{L^2}+\|u\|^2_{L^2}\|\tau\|^2_{L^2}ds$ and $\tilde{B}_2=\int_0^t\|\nabla u\|_{L^2}\|\tau\|^2_{L^2}ds$.
		Integrating \eqref{53ineq1} on $[0,t]$, we infer that
		\begin{align}\label{53ineq2}
			(1+t)^2E_{\eta} \lesssim (1+t)\|(u_0,\tau_0)\|_{\dot{B}^{-1}_{2,\infty}} + (1+t)B_1+(1+t)^{\frac{3}{2}}\tilde{B}_2.
		\end{align}
		Denote ${\rm Q}(t) = \mathop{\max}\limits_{s\in[0,t]} (1+t)^{\frac{1}{2}}E_0(s)$. Then we deduce from \eqref{53ineq2} and Proposition \ref{5prop2} with $l=4$ that
		\begin{align}\label{53ineq3}
			Q(t) &\lesssim \|(u_0,\tau_0)\|_{\dot{B}^{-1}_{2,\infty}} + \int_0^t \frac{\ln^{-2}(e+s)}{1+s}Q(s)ds \\ \notag
			&~~~+\int_0^t(1+s)^{-\frac{1}{2}}\|\nabla u\|_{L^2}Q(s)ds.
		\end{align}
		Applying Gronwall's inequality to \eqref{53ineq3}, we obtain
		\begin{align*}
			Q(t) &\lesssim \|(u_0,\tau_0)\|_{\dot{B}^{-1}_{2,\infty}} \exp\{\int_0^t(1+s)^{-\frac{1}{2}}\|\nabla u\|_{L^2}ds\}\\
			&\lesssim \|(u_0,\tau_0)\|_{\dot{B}^{-1}_{2,\infty}} \exp\{\left(\int_0^t\frac{\ln^{-2}(e+s)}{1+s}ds\right)^{\frac{1}{2}} \left(\int_0^t\ln^{2}(e+s)\|\nabla u\|^2_{L^2}ds \right)^{\frac{1}{2}}\} \\
			&\lesssim 1,
		\end{align*}
		which implies that
		\begin{align}\label{53ineq4}
		E_0 \lesssim (1+t)^{-\frac{1}{2}}.
	    \end{align}
	
		Choosing $f(t)=(1+t)^{\frac{3}{2}}$ in \eqref{52ineq1} and using \eqref{53ineq4}, we obtain
		\begin{align}\label{53ineq5}
			\frac{d}{dt}\left((1+t)^{\frac{3}{2}}E_{\eta}\right) + \frac 1 2(1+t)^{\frac{3}{2}}H_{\eta}  &\lesssim (1+t)^{-\frac{1}{2}}\left(\|(u_0,\tau_0)\|_{\dot{B}^{-1}_{2,\infty}}+B_1\right) + \tilde{B}_2 \\ \notag
			&\lesssim 1 +\int_0^t(1+s)^{-\frac{1}{2}}\|\nabla u\|_{L^2}ds \\ \notag
			&\lesssim 1.
		\end{align}
		Integrating \eqref{53ineq5} over $[0,t]$, we infer that
		\begin{align*}
			(1+t)^{\frac{3}{2}}E_{\eta} + \int_0^t(1+s)^{\frac{3}{2}}H_{\eta}ds \lesssim 1+t,
		\end{align*}
		which implies that
		\begin{align}\label{53ineq6}
		(1+t)^{-1}\int_0^t(1+s)^{\frac{3}{2}}E_1ds \lesssim 1.
		\end{align}
		We next to derive improved time decay rate for $E_1$. Multiplying \eqref{5ineq2} by $(1+t)^{2}$ and using \eqref{53ineq4}, one can arrive at
		\begin{align}\label{53ineq7}
			\frac{d}{dt}\left((1+t)^{2}E_1\right) + (1+t)^{2}\|\nabla^2\tau\|^2_{L^2} &\lesssim (1+t)^{2}\|\nabla u\|^2_{L^2}\|\tau\|^2_{H^1} + (1+t)E_1 \\ \notag
			&\lesssim (1+t)^{\frac{3}{2}}E_1.
		\end{align}
		Integrating \eqref{53ineq7} over $[0,t]$, we infer from \eqref{53ineq6} that
		\begin{align*}
			(1+t)E_1 \lesssim (1+t)^{-1}\int_0^t(1+s)^{\frac{3}{2}}E_1ds \lesssim 1,
		\end{align*}
		which implies $E_1 \lesssim (1+t)^{-1}$. We thus complete the proof of Proposition \ref{5prop3}.
	\end{proof}
	The following lemma reveals that the algebraic decay rate can provide information for the boundedness of solutions in Besov spaces with negative indicator.
	\begin{lemm}\label{5lemma2}
		Let $\gamma,\sigma\in[0,1]$ and $\sigma < 3-\frac{1}{\gamma}$. Assume that $(u,\tau)$ satisfies the conditions in Theorem \ref{theo2}. If the following conditions are satisfied:\\
		(1) decay condition
		\begin{align}\label{54ineq1}
			E^2_0 + E_1 \leq C(1+t)^{-2\gamma},
		\end{align}
		(2) integral condition
		\begin{align}\label{54ineq2}
			\int_0^t (1+s)^{\beta}E_1ds \leq C,~~~~\forall~\beta\in[0,\gamma),
		\end{align}
		then we obtain
		$$
		M_{\sigma}=\mathop{\max}\limits_{t\in[0,\infty)}\|(u,\tau)\|_{\dot{B}^{-\sigma}_{2,\infty}}(t)\leq C.
		$$
	\end{lemm}
	\begin{proof}
		Applying $\dot{\Delta}_j$ to system \eqref{eq4}, we get
		\begin{align}\label{54ineq3}
			\left\{\begin{array}{l}
				\partial_t\dot{\Delta}_ju - \dot{\Delta}_j{\rm div}~\tau = -\dot{\Delta}_j(u\cdot\nabla u)-\dot{\Delta}_j\nabla {\rm P}  ,\\[1ex]
				\partial_t\dot{\Delta}_j\tau - \dot{\Delta}_j\Delta\tau-\dot{\Delta}_j D(u)= -\dot{\Delta}_j(u\cdot\nabla\tau)-\dot{\Delta}_jQ(\nabla u,\tau).
			\end{array}\right.
		\end{align}
		Taking $L^2$ inner product for $(\ref{54ineq3})_1$ with $\dot{\Delta}_ju$, we have
		\begin{align}\label{54ineq4}
			\frac{1}{2}\frac{d}{dt}\|\dot{\Delta}_ju\|^2_{L^2} - \langle \dot{\Delta}_j{\rm div}~\tau,\dot{\Delta}_ju\rangle \leq \|\dot{\Delta}_j(u\cdot\nabla u)\|_{L^2}\|\dot{\Delta}_ju\|_{L^2}.
		\end{align}
		Taking $L^2$ inner product for $(\ref{54ineq3})_2$ with $\dot{\Delta}_j\tau$, we obtain
		\begin{align}\label{54ineq5}
			\frac{1}{2}\frac{d}{dt}\|\dot{\Delta}_j\tau\|^2_{L^2} + \|\nabla\dot{\Delta}_j\tau\|^2_{L^2}-\langle \dot{\Delta}_jD(u),\dot{\Delta}_j\tau\rangle &\leq \|\dot{\Delta}_j(u\otimes\tau)\|_{L^2}\|\nabla\dot{\Delta}_j\tau\|_{L^2} \\ \notag
			&~~~+ \|\dot{\Delta}_jQ(\nabla u,\tau)\|_{L^2}\|\dot{\Delta}_j\tau\|_{L^2} \\ \notag
			&\leq \|\dot{\Delta}_j(u\otimes\tau)\|^2_{L^2}+\frac 1 2\|\nabla\dot{\Delta}_j\tau\|^2_{L^2} \\ \notag
			&~~~+ \|\dot{\Delta}_jQ(\nabla u,\tau)\|_{L^2}\|\dot{\Delta}_j\tau\|_{L^2}.
		\end{align}
		According to the symmetric of $\tau$, we obtain the following cancellation
		$$
		\langle \dot{\Delta}_j{\rm div}~\tau,\dot{\Delta}_ju\rangle + \langle \dot{\Delta}_jD(u),\dot{\Delta}_j\tau\rangle = 0,
		$$
		from which we can deduce from \eqref{54ineq4} and \eqref{54ineq5} that
		\begin{align}\label{54ineq6}
			\frac{d}{dt}\|\dot{\Delta}_j(u,\tau)\|^2_{L^2}
			&\lesssim \|\dot{\Delta}_j(u\cdot\nabla u)\|_{L^2}\|\dot{\Delta}_ju\|_{L^2} +  \|\dot{\Delta}_j(u\otimes\tau)\|^2_{L^2} + \|\dot{\Delta}_jQ(\nabla u,\tau)\|_{L^2}\|\dot{\Delta}_j\tau\|_{L^2}.
		\end{align}
		Multiplying \eqref{54ineq6} by $2^{-2\sigma j}$ and taking the maximum for $j \in  \mathbb{Z}$, we conclude that
		\begin{align}\label{54ineq7}
			\frac{d}{dt}\|(u,\tau)\|^2_{\dot{B}^{-\sigma}_{2,\infty}}
			&\lesssim \|u\cdot\nabla u\|_{\dot{B}^{-\sigma}_{2,\infty}}\|u\|_{\dot{B}^{-\sigma}_{2,\infty}} +  \|u\otimes\tau\|^2_{\dot{B}^{-\sigma}_{2,\infty}} + \|Q(\nabla u,\tau)\|_{\dot{B}^{-\sigma}_{2,\infty}}\|\tau\|_{\dot{B}^{-\sigma}_{2,\infty}}.
		\end{align}
		Define ${\rm M}(t) = \mathop{\max}\limits_{s\in[0,t]}\|(u,\tau)\|^2_{\dot{B}^{-\sigma}_{2,\infty}}(s)$. Integrating \eqref{54ineq7} over $[0,t]$, we infer that
		\begin{align}\label{54ineq8}
			M^2(t) &\lesssim M^2(0) + \int_0^t\|u\otimes\tau\|^2_{\dot{B}^{-\sigma}_{2,\infty}} ds + M(t)\int_0^t\|u\cdot\nabla u\|_{\dot{B}^{-\sigma}_{2,\infty}} + \|Q(\nabla u,\tau)\|_{\dot{B}^{-\sigma}_{2,\infty}}ds.
		\end{align}
		Applying embedding $L^{\frac{2}{\sigma+1}} \hookrightarrow \dot{B}^{-\sigma}_{2,\infty}$ and H\"older's inequality, one can arrive at
		\begin{align}\label{54ineq9}
			\int_0^t\|u\otimes\tau\|^2_{\dot{B}^{-\sigma}_{2,\infty}} ds
			&\lesssim  \int_0^t\|u\otimes\tau\|^2_{L^{\frac{2}{\sigma+1}}} ds \\ \notag
			&\lesssim  \int_0^t\|u\|^2_{L^{\frac{2}{\sigma}}}\|\tau\|^2_{L^2} ds \\ \notag
			&\lesssim  \int_0^t\|u\|^{2\sigma}_{L^2}\|\nabla u\|^{2-2\sigma}_{L^2}\|\tau\|^2_{L^2} ds \\ \notag
			&\lesssim \int_0^t (1+s)^{-3\gamma+\sigma\gamma} ds\\ \notag
			&\lesssim 1,
		\end{align}
		where $\sigma < 3-\frac{1}{\gamma}$. Moreover, we deduce from the conditions \eqref{54ineq1} and \eqref{54ineq2} that
		\begin{align}\label{54ineq10}
			\int_0^t\|u\cdot\nabla u\|_{\dot{B}^{-\sigma}_{2,\infty}} ds &\lesssim  \int_0^t\|u\cdot\nabla u\|_{L^{\frac{2}{\sigma+1}}} ds \\ \notag
			&\lesssim  \int_0^t\|u\|_{L^{\frac{2}{\sigma}}}\|\nabla u\|_{L^2} ds \\ \notag
			&\lesssim  \int_0^t\|u\|^{\sigma}_{L^2}\|\nabla u\|^{1-\sigma}_{L^2}\|\nabla u\|_{L^2} ds \\ \notag
			&\lesssim \left(\int_0^t (1+s)^{-2\gamma+\sigma\gamma-\beta} ds\right)^{\frac{1}{2}}\left(\int_0^t (1+s)^{\beta}\|\nabla u\|^2_{L^2}ds\right)^{\frac{1}{2}}\\ \notag
			&\lesssim 1,
		\end{align}
		where we take $(\sigma-2)\gamma+1<\beta<\gamma$.
		Similarly, we have
		\begin{align}\label{54ineq11}
			\int_0^t\|Q(\nabla u,\tau)\|_{\dot{B}^{-\sigma}_{2,\infty}} ds &\lesssim  \int_0^t\|Q(\nabla u,\tau)\|_{L^{\frac{2}{\sigma+1}}} ds \\ \notag
			&\lesssim  \int_0^t\|\tau\|_{L^{\frac{2}{\sigma}}}\|\nabla u\|_{L^2} ds \\ \notag
			&\lesssim  \int_0^t\|\tau\|^{\sigma}_{L^2}\|\nabla \tau\|^{1-\sigma}_{L^2}\|\nabla u\|_{L^2} ds \\ \notag
			&\lesssim \left(\int_0^t (1+s)^{-2\gamma+\sigma\gamma-\beta} ds\right)^{\frac{1}{2}}\left(\int_0^t (1+s)^{\beta}E_1ds\right)^{\frac{1}{2}}\\ \notag
			&\lesssim 1.
		\end{align}
		Combining estimates \eqref{54ineq8}-\eqref{54ineq11}, we conclude that
		\begin{align*}
			M^2(t) \lesssim 1+M(t),
		\end{align*}
		which implies that $M(t) \lesssim 1$. This completes the proof of Lemma \ref{5lemma2}.
	\end{proof}
	
	The following lemma reveals that the boundedness of solutions in Besov spaces with negative indicator can improve the algebraic decay rate given.
	\begin{lemm}\label{5lemma3}
		Let $\alpha,\gamma,\sigma\in[\frac{1}{2},1]$. Moreover, $\gamma < \frac{\sigma}{2}+\frac{1}{2}$ when $\alpha = \frac{1}{2}$ and $\gamma \leq \frac{\sigma}{2}+\frac{1}{2}$ when $\alpha > \frac{1}{2}$. Assume that $(u,\tau)$ satisfies the conditions in Theorem \ref{theo2} and \\
		(1) boundness condition
		\begin{align}\label{55ineq1}
			M_{\sigma}=\mathop{\max}\limits_{t\in[0,\infty)}\|(u,\tau)\|_{\dot{B}^{-\sigma}_{2,\infty}}(t)\leq C,
		\end{align}
		(2) decay condition
		\begin{align}\label{55ineq2}
			E^2_0 + E_1 \leq C(1+t)^{-2\alpha},
		\end{align}
		(3) integral condition
		\begin{align}\label{55ineq3}
			\int_0^t (1+s)^{\kappa}E_1ds \leq C,~~~~\forall~\kappa\in[0,\alpha),
		\end{align}
		then the following properties hold:\\
		(1) attenuation property :
		\begin{align}\label{55ineq4}
			E^2_0 + E_1 \leq C(1+t)^{-2\gamma},
		\end{align}
		(2) integral property :
		\begin{align}\label{55ineq5}
			\int_0^t (1+s)^{\beta}E_1ds \leq C,~~~\forall~\beta\in[0,\gamma).
		\end{align}
	\end{lemm}
	\begin{proof}
		We infer from Proposition \ref{prop0} and boundness condition \eqref{55ineq1} that
		\begin{align}\label{55ineq6}
			\int_{S(t)}|\hat{\tau}|^2d\xi
			&\lesssim \mathop{\sum}\limits_{j\leq \log_2[\frac{4}{3}C^{\frac{1}{2}}_2(1+t)^{-\frac{1}{2}}]}\|\dot{\Delta}_j\tau\|^2_{L^2} \\ \notag
			&\lesssim  \mathop{\sum}\limits_{j\leq \log_2[\frac{4}{3}C^{\frac{1}{2}}_2(1+t)^{-\frac{1}{2}}]}2^{2\sigma j}M_{\sigma}^2 \\ \notag
			& \lesssim (1+t)^{-\sigma}.
		\end{align}
		Therefore, we deduce from \eqref{55ineq6} and $f(t)=1+t$ that
		\begin{align}\label{55ineq7}
			B_2	&\lesssim (1+t)^{-\frac{1}{2}}\int_0^t\|\nabla u\|_{L^2}\|\tau\|_{L^2} \left(\int_{S(t)}|\hat{\tau}|^2d\xi\right)^{\frac{1}{2}} ds \\ \notag
			&\lesssim (1+t)^{-\frac{\sigma+1}{2}}\int_0^t\|\nabla u\|_{L^2}\|\tau\|_{L^2} ds.
		\end{align}
		Therefore, by \eqref{51ineq0}, one can arrive at
		\begin{align}\label{55ineq8}
			\frac{d}{dt}E_{\eta} + H_{\eta} + \frac{C_2(\eta+1)}{16(1+t)}\|(u,\tau)\|^2_{L^2}
			&\lesssim (1+t)^{-2}\left(\|(u_0,\tau_0)\|^2_{\dot{B}^{-1}_{2,\infty}} + B_1\right) + (1+t)^{-1}B_2
			\\ \notag&\lesssim (1+t)^{-2}\left(\|(u_0,\tau_0)\|^2_{\dot{B}^{-1}_{2,\infty}}+B_1\right) \\ \notag
			&~~~+(1+t)^{-\frac{\sigma+3}{2}}\int_0^t\|\nabla u\|_{L^2}\|\tau\|_{L^2} ds.
		\end{align}
		Multiplying \eqref{55ineq8} by $(1+t)^2$ and taking $C_2$ and $t$ sufficiently large, we infer that
		\begin{align}\label{55ineq9}
			\frac{d}{dt}\left((1+t)^2E_{\eta}\right)
			&\lesssim \|(u_0,\tau_0)\|_{\dot{B}^{-1}_{2,\infty}}+\int_0^t \|u\|^3_{L^2}+\|u\|^2_{L^2}\|\tau\|^2_{L^2}ds \\ \notag
			&~~~+ (1+t)^{-\frac{\sigma}{2}+\frac{1}{2}}\int_0^t\|\nabla u\|_{L^2}\|\tau\|_{L^2} ds.
		\end{align}
		Integrating \eqref{55ineq9} over $[0,t]$ and multiplying by $(1+t)^{\gamma-2}$, we obtain
		\begin{align}\label{55ineq10}
			(1+t)^{\gamma}E_{\eta}
			&\lesssim (1+t)^{\gamma-1}\|(u_0,\tau_0)\|_{\dot{B}^{-1}_{2,\infty}}+(1+t)^{\gamma-1}\int_0^t \|u\|^3_{L^2}+\|u\|^2_{L^2}\|\tau\|^2_{L^2}ds \\ \notag
			&~~~+ (1+t)^{\gamma-\frac{\sigma}{2}-\frac{1}{2}}\int_0^t\|\nabla u\|_{L^2}\|\tau\|_{L^2} ds.
		\end{align}
		Denote $Q(t) = \mathop{\max}\limits_{s\in[0,t]}(1+s)^{\gamma}E_{\eta}(s)$. Thus, we infer from conditions \eqref{55ineq2}, \eqref{55ineq3} and \eqref{55ineq10} that
		\begin{align}\label{55ineq11}
			Q(t)
			&\lesssim \|(u_0,\tau_0)\|_{\dot{B}^{-1}_{2,\infty}}+\int_0^t\frac{\ln^{-2}(e+s)}{1+s}Q(s)ds \\ \notag
			&~~~+ \int_0^t(1+s)^{\gamma-\frac{\sigma}{2}-\frac{1}{2}}\|\nabla u\|_{L^2}\|\tau\|_{L^2} ds \\ \notag
			& \lesssim \|(u_0,\tau_0)\|_{\dot{B}^{-1}_{2,\infty}}+\int_0^t\frac{\ln^{-2}(e+s)}{1+s}Q(s)ds \\ \notag
			&~~~+ \left(\int_0^t(1+s)^{2\gamma-\sigma-1-\alpha-\kappa}ds\right)^{\frac{1}{2}}\left(\int_0^t(1+s)^{\kappa}\|\nabla u\|^2_{L^2}ds\right)^{\frac{1}{2}} \\ \notag
			& \lesssim 1 + \int_0^t\frac{\ln^{-2}(e+s)}{e+s}Q(s)ds.
		\end{align}
		Applying Gronwall's inequality, we obtain $Q(t) \lesssim 1$, which implies that
		$$
		E_0 \lesssim (1+t)^{-\gamma}.
		$$
		
		By performing a routine procedure, one can arrive at
		\begin{align}\label{55ineq12}
			\int_0^t (1+s)^{\beta}E_1 \leq C(\beta),
		\end{align}
		for any $\beta\in[0,\gamma)$ and
		\begin{align}\label{55ineq13}
			(1+t)^{-1}\int_0^t (1+s)^{\gamma+1}E_1ds \lesssim 1.
		\end{align}
		We next to derive improved decay rate for $E_1$. Multiplying \eqref{5ineq2} by $(1+t)^{2\gamma+1}$, one can arrive at
		\begin{align}\label{55ineq14}
			\frac{d}{dt}\left((1+t)^{2\gamma+1}E_1\right) &\lesssim (1+t)^{2\gamma+1}\|\nabla u\|^2_{L^2}\|\tau\|^2_{H^1} + (1+t)^{2\gamma}E_1 \\ \notag
			&\lesssim (1+t)^{\gamma+1}E_1.
		\end{align}
		Integrating \eqref{55ineq14} over $[0,t]$ and dividing by $1+t$, we infer that
		\begin{align}\label{55ineq15}
			(1+t)^{2\gamma}E_1 \lesssim (1+t)^{-1}\int_0^t(1+s)^{\gamma+1}E_1ds,
		\end{align}
		which implies that $E_1 \lesssim (1+t)^{-2\gamma}$. We thus complete the proof of Lemma \ref{5lemma3}.
	\end{proof}
	\textbf{Proof of Theorem \ref{theo2} :}\\
	Firstly, we infer from Proposition \ref{5prop3} and Lemma \ref{5lemma2} by taking $\gamma=\frac{1}{2}$ and $\sigma=\frac{4}{5}$ that
	$$(u,\tau) \in \dot{B}^{-\frac{4}{5}}_{2,\infty}.$$
	Let $\alpha = \frac{1}{2}$ and $\sigma = \gamma = \frac{4}{5}$, we deduce from Lemma \ref{5lemma3} that
	$$E^2_0 + E_1 \lesssim (1+t)^{-\frac{4}{5}}.$$
	Again, we infer from Lemma \ref{5lemma2} by taking $\gamma=\frac{4}{5}$ and $\sigma = 1$ that
	$$(u,\tau) \in \dot{B}^{-1}_{2,\infty}\times \dot{B}^{-1}_{2,\infty}.$$
	Let $\alpha = \frac{4}{5}$, $\sigma = \gamma = 1$. Finally, we  deduce from Lemma \ref{5lemma3} that
	$$E^2_0 + E_1 \lesssim (1+t)^{-2}.$$
	This completes the proof of Theorem \ref{theo2}.
	\hfill$\Box$
	
\smallskip
\noindent\textbf{Acknowledgments} The authors would like to thank Prof. Zhen Lei for many stimulating discussions and constructive suggestions. This work was partially supported by the National Natural Science Foundation of China (No. 12171493), the Macao Science and Technology Development Fund (No. 0091/2018/A3), Guangdong Province of China Special Support Program (No. 8-2015), the China Postdoctoral Science Foundation (No. 2022TQ0077 and No. 2023M730699) and Shanghai Post-doctoral Excellence Program (No. 2022062).	
%The authors thank the referee for valuable comments and suggestions.

\phantomsection
\addcontentsline{toc}{section}{\refname}
%添加参考文献到书签，宏包 hyperref
\bibliographystyle{abbrv} %plain ,%alpha, %abbrv
\bibliography{Oldroyd-Bref}

\end{document}